\documentclass[11pt]{amsart}
\usepackage[T1]{fontenc}
\usepackage{lmodern}
\usepackage{ifthen}
\usepackage{amsfonts}
\usepackage{amsxtra}
\usepackage{amssymb}
\usepackage{amsthm}
\usepackage{array}
\usepackage[margin=1in]{geometry}
\usepackage{xcolor}
\definecolor{cite}{rgb}{0.50,0.00,1.00}
\definecolor{url}{rgb}{0.00,0.50,0.75}
\definecolor{link}{rgb}{0.00,0.00,0.50}
\usepackage[colorlinks,linkcolor=link,urlcolor=url,citecolor=cite,pagebackref,breaklinks]{hyperref}
\usepackage{mathtools}
\usepackage{mathrsfs}
\usepackage[all]{xy}
\usepackage[lite,abbrev,msc-links]{amsrefs}

\makeindex

\theoremstyle{definition} 
\newtheorem{Unity}{Unity}[section] 
\newtheorem*{Definition*}{Definition} 
\newtheorem{Definition}[Unity]{Definition} 

\newtheorem*{Setting}{Setting}

\theoremstyle{plain} 
\newtheorem*{Theorem*}{Theorem}
\newtheorem{Theorem}[Unity]{Theorem}
\newtheorem{Proposition}[Unity]{Proposition}
\newtheorem*{Corollary*}{Corollary}
\newtheorem{Corollary}[Unity]{Corollary}
\newtheorem{Lemma}[Unity]{Lemma}

\theoremstyle{remark} 
\newtheorem*{Remark*}{Remark}

\numberwithin{Unity}{section}

\newcommand{\X}{\mathcal{X}}
\newcommand{\Y}{\mathcal{Y}}
\newcommand{\M}{\mathcal{M}}
\newcommand{\C}{\mathcal{C}}
\newcommand{\CL}{\mathcal{CL}}
\newcommand{\FL}{\mathcal{FL}}

\begin{document}

\title{Equivalence of two notions of log moduli stacks}
\author{Junchao Shentu}
\address{Academy of Mathematics and Systems Science, Chinese Academy of Science, Beijing, P. R. China}
\email{stjc@amss.ac.cn}
\begin{abstract}
We show the equivalence between two notions of log moduli stacks which appear in literatures. In particular, we generalize M.Olsson's theorem of representation of log algebraic stacks and answer a question posted by him (\cite{Ol4} 3.5.3). As an application, we obtain several fundamental results of algebraic log stacks which resemble to those in algebraic stacks.
\end{abstract}
\maketitle
\tableofcontents
\section{Introduction}
The Logarithmic structure is initiated by J.M.Fontaine and L.Illusie to treat various degenerations in algebraic geometry, which is further developed by Kazuya Kato \cite{KKato}. It is also realized that the logarithmic structure naturally appears on the boundary in various moduli spaces, such as $\overline{\mathcal{M}}_{g,n}$, the moduli stack of stable curves of genus $g$ with $n$ marked points \cite{FKato}, moduli space of principally polarized abelian varieties \cite{Ol2}, and moduli of polarized K3 surfaces \cite{Ol3}. For a survey of logarithmic geometry and further references, we refer to \cite{Ambro}.

However, there are two reasonable ways to consider log stacks:
\begin{enumerate}
  \item One way is to take an algebraic stack (Definition \ref{stack}) $\X$ with fppf topology, and define the log structure on $\X$ as we did for schemes. Namely, a log structure on $\X$ is a pair $(\M, \alpha)$ where $\M$ is a coherent sheaf of monoid, $\alpha:\M\rightarrow \mathcal{O}_{\X}$ is a homomorphism of monoids to multiply monoid of $\mathcal{O}_{\X}$, and $\alpha|_{\alpha^{-1}\mathcal{O}_{\X}^\ast}:\alpha^{-1}\mathcal{O}_{\X}^\ast\rightarrow \mathcal{O}_{\X}^\ast$ is an isomorphism. We call such algebraic stack with log structure a \emph{log algebraic stack} (Definition \ref{DEF_logalgstack}).
  \item Another way is to consider stacks on the category of fine log schemes (with fppf topology). We call it \emph{algebraic log stack} (Definition \ref{DEF_Alglogstack}) if: (1) its diagonal is Asp-representable (Definition \ref{DEF_ab_rep_Asp}); (2) it has a strict log smooth cover by a fine log scheme. This notion of algebraic log stack is considered in \cite{Ol4}, with some finite presentation condition. We consider general algebraic log stacks in this paper.
\end{enumerate}

The notion of log algebraic stacks is theoretically easier to understand. Since the objects are algebraic stacks with log structures, we can study them based on the theory of algebraic stacks \cite{dejong} and log structures in \cite{KKato},  \cite{Ol1}. In practice, log algebraic stacks arise when moduli spaces have a codimension 1 boundary parameterizing degenerate objects.

However, when concerning moduli problems with degenerations (or compactification of moduli spaces), the notion of algebraic log stack is more natural, as it is presented in \cite{FKato}. For example, to get a compact moduli space of varieties, one has to consider families with degenerate fibers, which always turn out to be log smooth. This suggests that we should establish the theory of stacks on the category of log schemes, in order to make the logarithmic geometry applicable to moduli problems. In this paper, we show that the two notions of log moduli stacks are equivalent. As a consequence, we can reduce all natural problems about algebraic log stacks to the corresponding problems on algebraic stacks (with fine log structure).

A log algebraic stack $\X$ naturally induces an algebraic stack $\widetilde{\X}$ over $\mathbf{Flog}$, by defining $\widetilde{\X}(U)=Hom(U,\X)$ the category of log morphisms from $U$ to $\X$. We say that $\widetilde{\X}$ is represented by $\X$ (In this paper we will reformulate this notion by generalized Gillam's functor (Definition \ref{GC}), which is conceptually more convenient).

The first comparison example is due to F.Kato. This paper is motivated by understanding it.
\begin{Theorem*}(\cite{FKato} Theorem 4.5)
The algebraic log stack $\mathcal{L}\M_{g,n}$ of log smooth curve of type $(g,n)$ is represented by log algebraic stack $(\overline{\M}_{g,n},\partial\overline{\M}_{g,n})$, where $\partial\overline{\M}_{g,n}$ is the log structure associated to the NC divisor corresponding to nonsmooth stable curves.
\end{Theorem*}

One general result is due to M.Olsson, we state it in our terminology:
\begin{Theorem*}(\cite{Ol4} Theorem 1.3.8)
An Algebraic log stack locally of finite presentation is represented by a log algebraic stack.
\end{Theorem*}

We generalize this result to the stacks not necessarily locally of finite presentation (Theorem \ref{Phi_corres_Astack}), using the notion of minimal objects (Definition \ref{minimalob}), which is conceptually more natural. First we would like to explain how the notion of minimal object naturally appears.

Given an algebraic log stack $\widetilde{\X}$ induced by a log algebraic stack $\X$, i.e., $\widetilde{\X}(U)=Hom(U,\X)$, the category of log morphisms from $U$ to $\X$. Such category $\widetilde{\X}$ (over $\mathbf{Flog}$) satisfies that, given a base scheme $\underline{B}$ and a morphism $f:\underline{B}\rightarrow\underline{\X}$, there is a log structure on $\underline{B}$ which makes $f$ strict. We denote the associated log morphism $\tau_f$. Then $\tau_f$ is minimal in the sense that for any log morphism $\tau:B\rightarrow\X$ with the underlining morphism $f$, $\tau$ factors through $\tau_f$. It turns out that enough compatible minimal objects is sufficient for representability (\cite{Gi1} or Lemma \ref{GC}), and one of the main results in this paper is that algebraic log stacks always have enough compatible minimal objects, hence representable (Theorem \ref{alglogstackminimal}). This generalize M.Olsson's result.

In concrete moduli problems, `minimal objects' (Definition \ref{minimalob}) is initiated in various works of F.Kato (\emph{basic log curve} in \cite{FKato}), M.Gross and B.Siebert (\emph{basic stable log map} in \cite{Gross}), M.Olsson (\emph{distinguished object} in \cite{Ol3} and \emph{solid objects} in \cite{Ol4}), etc., and is formally studied by W.D.Gillam in \cite{Gi1}.

The main results of our paper are:

\begin{Theorem*}
Given a fine log scheme $S$. There is a canonical strict 2-functor (Definition \ref{DEF_2-functor}) $$\Phi_{log}^{alg}:\mathbf{LAS}/S\rightarrow\mathbf{ALS}/S$$ from the 2-strict category (Definition \ref{DEF_2-category}) of log algebraic stacks (Definition \ref{DEF_logalgstack}) to the 2-strict category of algebraic log stacks (Definition \ref{DEF_Alglogstack}, Proposition \ref{Phialgstack}), satisfying:
\begin{enumerate}
  \item $\Phi^{alg}_{log}$ restricts to a stricts 2-equivalence (Definition \ref{DEF_2-functor}, Theorem \ref{Phi_corres_Astack})
  $$\mathbf{FLAS}/S\rightarrow\mathbf{ALS}/S$$.

  \item $\X\in \mathbf{FLAS}_S$ is a fine log scheme (resp. a DM-log stack or an fine log algebraic space), (locally) Noetherian, regular, normal, $S_n$, Cohen-Macaulay, reduced, of characteristic $p$, saturated, log regular, quasi-compact, quasi-separate, etc. if and only if $\Phi^{alg}_{log}\X$ is (Theorem \ref{Phi_presevescheme}, \ref{Phialgsp}, \ref{Phi_preproperty_Astack}).
  \item A morphism $f$ in $\mathbf{FLAS}_S$ is representable, Asp-representable, locally of finite representation, flat, smooth, normal, Cohen-Macaulay, $S_n$, quasi-compact, quasi-separated, strict, integral, saturated, Kummer, Cartier, log smooth, log flat, etc. if and only if $\Phi^{alg}_{log}f$ is (Theorem \ref{Phi_premorph_Astack}).
\end{enumerate}

\end{Theorem*}

As consequences of the corresponding above, we have:
\begin{Theorem*}(Theorem \ref{fppfrepofstack})
Let S be a fine log scheme, $F:\X\rightarrow \Y$ be a 1-morphism of stacks over $\mathbf{Flog}_{S,fppf}$ . If
\begin{enumerate}
  \item $\X$ is a log algebraic space,
  \item $F$ is Asp-representable, strict, surjective, flat and locally of finite
  presentation,
\end{enumerate}
then $\Y$ is an algebraic log stack.
\end{Theorem*}

\begin{Theorem*}(Theorem \ref{stackgroupiod})
Let $S$ be a fine log scheme and $X$ be an algebraic log stack over $S$. $f : U\rightarrow X$ is a surjective strict log smooth morphism where $U$ is an algebraic log space over $S$. Let $(U; R; s; t; c)$ be the associated groupoid in log algebraic spaces and $f_{can}:[U/R]\rightarrow X$ be the associated map. Then
\begin{enumerate}
  \item the morphisms $s$, $t$ are strict log smooth, and
  \item the 1-morphism $f_{can}:[U/R]\rightarrow X$ is an equivalence.
\end{enumerate}
\end{Theorem*}\noindent
\textbf{Remark:} If the morphism $f:U\rightarrow X$ is only assumed surjective, strict, flat and locally of finite presentation, then it will still be the case that $f_{can}:[U/R]\rightarrow X$ is an equivalence. In this case the morphisms $s$, $t$ will be strict, flat and locally of finite presentation, but of course not smooth in general.

\begin{Theorem*}(Theorem \ref{groupoidstack})
Let $S$ be a fine log scheme and $(U; R; s; t; c)$ be a log smooth groupoid in algebraic log spaces over $S$. Then the quotient stack $[U/R]$ is an algebraic log stack over $S$.
\end{Theorem*}

\begin{Theorem*}(Theorem \ref{DM=unramify_diagonal})
An algebraic log stack $\X$ is DM if and only if the diagonal $\Delta_{\X}$ is unramified.
\end{Theorem*}

\begin{Corollary*}(Corollary \ref{Algspa=mono_diagonal})
For an algebraic log stack $\X$, the following are equivalent:
 \begin{enumerate}
   \item $\X$ is an algebraic log space;
   \item for every $x\in\X_U$, where $U\in\mathbf{Flog}_S$, $Aut_{\X_U}(x)=\{id_x\}$;
   \item the diagonal $\Delta_{\X}:\X\rightarrow\X\times_{\X}\X$ is fully faithful.
 \end{enumerate}
\end{Corollary*}

\begin{Corollary*}(Corollary \ref{criterion_Asprep})
A morphism $f:\X\rightarrow\Y$ in $\mathbf{ALS}_S$ is Asp-representable if and only if $\Delta_f:\X\rightarrow\X\times_{\Y}\X$ is a monomorphism.
\end{Corollary*}

Although our main interest is the representability of algebraic log stack, we have to establish the fundamental results from the beginning. Most of the notions and results in this paper are natural, but lack of references.

The paper is organized as follows:

In section 2, we fix some categorical notions and generalize Gillam's functor \cite{Gi1} for our use. The main result is Proposition \ref{abGC}. As a corollary, given a fine log scheme $S$, we prove that there is a correspondence between stacks over $\mathbf{Sch}_{\underline{S}}$ with log structures and stacks over $\mathbf{Flog}_S$ with enough compatible minimal objects (Corollary \ref{logGC}).

In section 3, we introduce two definitions on log moduli algebraic spaces and morphisms in both categories. We show that the two categories of fine log moduli algebraic spaces are equivalent (Theorem \ref{Phi_corres_Asp}), and the equivalence respects various properties (Proposition \ref{Phi_premorph_Asp}, Proposition \ref{Phi_preproperty_Asp}).

In section 4, we introduce various definitions on log moduli stacks and morphisms in both categories. We show that the two categories of fine log moduli stack are equivalent (Theorem \ref{Phi_corres_Astack}), and the equivalence respects various properties (Proposition \ref{Phi_premorph_Astack}, Proposition \ref{Phi_preproperty_Astack}).

In section 5, we state several fundamental results in algebraic log stacks, as an application of the correspondence established in section 3 and 4. In particular, we prove the bootstrapping theorem on algebraic log stacks (Theorem \ref{fppfrepofstack}), the theorem on presentation of algebraic log stack by groupoid in log algebraic spaces (Theorem \ref{stackgroupiod}, Theorem \ref{groupoidstack} ), the criterion for an algebraic log stack to be DM (Theorem \ref{DM=unramify_diagonal}) or to be algebraic log space (Corollary \ref{Algspa=mono_diagonal}) and the criterion for morphism to be Asp-representation (Corollary \ref{criterion_Asprep}).

\textbf{Notations:}

\begin{itemize}
  \item We use log (structure, scheme, stack. etc.) to mean coherent log (structure, scheme, stack. etc.), unless otherwise mentioned.
  \item The log structure in this paper is always given on fppf topology (for fine log structure, it's equivalent to use \'etale topology if we consider DM-stack, by Theorem A.1 in \cite{Ol1}).
  \item We use capital letters $X$, $Y$, $S$ etc. to denote geometrical objects with no nontrivial automorphisms (e.g. algebraic log spaces or log schemes), and $\X$, $\mathcal{Y}$ otherwise (e.g. log algebraic stacks or generally groupoid fibered categories). We use underlined letters $\underline{X}$, $\underline{\X}$, etc. to denote the underlining spaces (stacks) and $M_X$ or $\M_\X$ to denote the log structures.
  \item Given a fine log scheme $S$, we denote $\mathbf{Flog}_S$ the category of fine log schemes over $S$, and $\mathbf{Clog}_S$ the category of coherent log schemes over $S$.
  \item We use \emph{stack} to mean groupoid fibration with all descent data effective, and use \emph{general stack} to mean fibration (not necessarily fibered in groupoid) with all descent data effective, although the latter is called `stack' in \cite{Vistoli}. We use $\simeq$ to denote the equivalence between objects in a strict 2-category. More categorical notions are introduced in section 2.
\end{itemize}

\section{Generalized Gillam's Functor}

In this section, we generalize Gillam's functor in \cite{Gi1}. First we recall some notions of categories.

\subsection{Categorical Notions}

\begin{Definition}\label{DEF_2-category}
A strict 2-category $\mathcal{C}$ is $(\mathcal{C}, \circ, \cdot, \star)$ consists of the following:
\begin{description}
  \item[a] a set of objects $Obj\mathcal{C}$;
  \item[b] for each $x, y\in\mathcal{C}$, A category $\mathcal{C}(x,y)$. The objects of $\mathcal{C}(x,y)$ are called 1-morphisms of $\mathcal{C}$. The morphisms of $\mathcal{C}(x,y)$ are called 2-morphisms of $\mathcal{C}$. We use $\cdot$ to denote the composition of 2-morphisms in $\mathcal{C}(x,y)$.
  \item[c] for each $x,y,z\in\mathcal{C}$, a composition functor
  $$m_{x,y,z}=(\circ, \star): \mathcal{C}(x,y)\times\mathcal{C}(y,z)\rightarrow\mathcal{C}(x,z)$$ sending object $(f,g)\in Obj(\mathcal{C}(x,y)\times\mathcal{C}(y,z))$ to $f\circ g\in Obj(\mathcal{C}(x,z))$ (in practice we always use $gf$ instead of $f\circ g$), and morphism $(\alpha, \beta)$ in $\mathcal{C}(x,y)(f_1,g_1)\times\mathcal{C}(y,z)(f_2,g_2)$ to $\alpha\star\beta\in\mathcal{C}(x,z)(f_1\circ f_2,g_1\circ g_2)$.
  \item[d] for each $x\in Obj\mathcal{C}$, a distinguished identity object $id_x\in\mathcal{C}(x,x)$.
\end{description}
satisfying:
\begin{enumerate}
  \item $\forall x,y,z,w\in Obj\mathcal{C}$, $m_{x,z,w}(m_{x,y,z}\times id_{\mathcal{C}(z,w)})=m_{x,y,w}(id_{\mathcal{C}(x,y)}\times m_{y,z,w})$;
  \item $\forall x,y\in Obj\mathcal{C}$ and 1-morphism $f\in \mathcal{C}(x,y)$, $id_x\circ f=f\circ id_y=f$;
  \item $\forall x,y\in Obj\mathcal{C}$ and 2-morphism $\alpha\in \mathcal{C}(x,y)$, $id_{id_x}\star\alpha=\alpha\star id_{id_y}=\alpha$;
\end{enumerate}
two objects $x,y\in Obj\C$ are said to be equivalent if there are 1-morphisms $f:x\rightarrow y$, $g:y\rightarrow x$ and 2-isomorphisms $\alpha:gf\simeq id_x$, and $\beta:fg\simeq id_y$.
\end{Definition}
\textbf{Remark:} There is a wider notion 'bicategory', but we won't need this.

A typical example of strict 2-category is the category of small categories $(\mathbf{Cat},\circ_{\mathbf{Cat}},\cdot_{\mathbf{Cat}},\bullet_{\mathbf{Cat}})$.
\begin{Definition}\label{DEF_2-functor}
A strict 2-functor $T:\C_1\rightarrow\C_2$ between two strict 2-categories consists of the following:
\begin{description}
  \item[a] a map $T:Obj(\C_1)\rightarrow Obj(\C_2)$;
  \item[b] for every $x,y\in\C_1$, a functor $T:\C_1(x,y)\rightarrow\C_2(Tx,Ty)$.
\end{description}
satisfying: $T(id_x)=id_{Tx}$, $T(id_{id_x})=id_{id_{Tx}}$, $T(f\circ g)=T(f)\circ T(g)$, $T(\alpha\star\beta)=T(\alpha)\star T(\beta)$. Moreover, $T$ is \emph{strictly fully faithful} if for any $x,y\in\C_1$, $T:\C_1(x,y)\rightarrow\C_2(Tx,Ty)$ is an isomorphism.
$y\in Obj\C_2$ is said to be \emph{in the essential image of $T$} if there is an object $x\in Obj\C_1$, s.t. $Tx$ is equivalent to $y$.
$T$ is said to be \emph{strict equivalent} if $T$ is strict fully faithful and every object in $\C_2$ is in the essential image of $T$.
\end{Definition}
\textbf{Remarks:} $T$ is a strict equivalence if and only if there is an 'inverse' pseudo-functor. However, we won't mention the definition of \emph{`inverse' pseudo-functor} for the complexity of 2-diagrams, and the ad hoc definition of strict equivalence above is enough for our purpose.

\subsection{Abstract Generalized Gillam's functor}

\begin{Setting}\
\begin{description}
  \item[S1] Fix a triple $(\C,\FL,\CL)$ with fibered category $\CL\rightarrow\C$ ($(f:x\rightarrow y)\mapsto(\underline{f}:\underline{x}\rightarrow\underline{y})$), with the associate groupoid fibration $\CL^\ast\rightarrow\C$, and a full sub-fibration $\FL\rightarrow\C$ with the associate groupoid fibration $\FL^\ast\rightarrow\C$;
  \item[S2] The inclusion functor $i:\FL\rightarrow\CL$ has a right adjoint $int:\CL\rightarrow\FL$. Both preserve cartesian arrows;
  \item[S3] $\C$, $\CL$ have fiber products, and the functor $\CL\rightarrow\C$ preserve fiber products;
  \item[S4] $\C$ is endowed with a Grothendieck topology $\tau$. We denote the associate site $\C_\tau$. Let $\CL_\tau$ (resp. $\FL_\tau$) be the site with pullback topology of $\tau$ through $\CL\rightarrow\C$ (resp. $\FL\rightarrow\C$), i.e. a covering in $\CL$ (resp. $\FL$) is $\{c_i\rightarrow c\}$ s.t. $\{\underline{c_i}\rightarrow \underline{c}\}$ is a covering in $\C_\tau$. $\FL\rightarrow\C$ is a general stack, and for every $x\in\CL$, the presheaf
      $$h_x(u)=\CL(u,x)$$ $u\in\FL$
      is a sheaf over $\FL$.
\end{description}
\end{Setting}
  Although we do abstractly in this section, the main situation we are interested in is $\mathbf{CLog}_S\rightarrow \mathbf{Sch}_{\underline{S}}$ with sub-fibration $\mathbf{FLog}_S\rightarrow \mathbf{Sch}_{\underline{S}}$ where $S$ is a log scheme.

\textbf{Remark about the setting:} For application one may want to consider fine log structures throughout the whole paper (log structure that is not fine is too much pathological). However, the difficulty arises when we take fiber product of fine log algebraic stack, which is not compatible with the fiber product of base algebraic stacks. This phenomena already appears in the category of log scheme, that is why K.Kato introduced the functor $int$ in \cite{KKato}, to make coherent log scheme integral. However, things works better in the larger category of coherent log schemes. So we introduce $\mathbf{CLog}_S$.

We form the strict 2-category $\mathbf{LogCat}/\C$ of categories over $\C$ with log structures:
\begin{itemize}
  \item the objects are pairs $(T, \M_{\X})$ of funcotrs $T:\X\rightarrow\C$, $\M_{\X}:\X\rightarrow\CL^\ast$ s.t. $T=\underline{\M_{\X}}$. $\M_{\X}$ is called the \emph{log structure} of $(T, \M_{\X})$. We use $\M_{\X}$ as Abbreviation of $(T, \M_{\X})$. $\M_{\X}$ is a fine log structure if $\M_{\X}$ factors through $\FL^\ast$;
  \item A 1-morphism is $$(F,F^\dag):(\M_{\X}:\X\rightarrow\CL)\rightarrow (\M_{\Y}:\Y\rightarrow\CL)$$ where $F:\X\rightarrow\Y$ is a functor with $\underline{\M_{\Y}F}=\underline{\M_{\X}}$, and $F^\dag:\M_{\X}\rightarrow \M_{\Y}F$ is a natural transformation in $\CL$ (not in $\CL^\ast$), with $\underline{F^\dag}=Id$;
  \item A 2-morphisms is $\eta:(F,F^\dag)\rightarrow(G,G^\dag)$ with $\eta:F\rightarrow G$ a natural transformation making the diagram of natural transformations
      $$\xymatrix{
       & \M_{\X} \ar[ld]_{F^\dag} \ar[rd]^{G^\dag} & \\
       \M_{\Y}F \ar[rr]^{\M_{\Y}\eta}& & \M_{\Y}G
      }$$
      commutes;
  \item for $\M_{\X}\in \mathbf{LogCat}/\C$, $Id_{\M_{\X}}$ is the identity functor, and $Id_{Id_{\M_{\X}}}$ is the identity transformation. $(F,F^\dag)\circ(G,G^\dag)=(F\circ_{\mathbf{Cat}} G,F^\dag\cdot_{\mathbf{Cat}} G^\dag_F)$, $\eta\star\tau=\eta\star_{\mathbf{Cat}}\tau$.
\end{itemize}

It can be verified that $\mathbf{LogCat}/\C$ is indeed a strict 2-category. We denote $\mathbf{FLogCat}/\C$ to be the full subcategory consisting of categories over $\C$ with a fine log structure $\M$. And we denote $\mathbf{LogCFG}/\C$ (resp: $\mathbf{FLogCFG}/\C$) the full subcategory of groupoid fibration over $\C$ with log structures (resp: with fine log structures).

Another strict 2-category we concern is the category of categories over $\FL$, we denote it $\mathbf{Cat}/{\FL}$:
\begin{itemize}
  \item The objects are functors $P_{\X}:\X\rightarrow\FL$;
  \item A 1-morphism $F:P_{\X}\rightarrow P_{\Y}$ is a functor $F:\X\rightarrow\Y$ with $P_{\Y}F=P_{\X}$;
  \item A 2-morphism $\eta:F\rightarrow G$ is a natural transformation $\eta:F\rightarrow G$ with $P_{\Y}\eta=id$.
  \item the identity and three composition operator are obvious.
\end{itemize}
We denote $\mathbf{CFG}/\FL$ the full subcategory of groupoid fibrations over $\FL$.

There is a natural strict 2-functor of strict 2-categories
$$\Phi:\mathbf{LogCat}/\C\rightarrow\mathbf{Cat}/\FL$$
as follows (which generalize notions in \cite{FKato} and \cite{Gi1}):
\begin{enumerate}
  \item for $(\M:\X\rightarrow\CL^\ast)$ the objects of $\Phi\M$ are pairs $(x, f)$, where $x\in\X$, $x'\in \FL$,  $f:x'\rightarrow \M x$ is a map in $\CL$ over $(\underline{\M x} ,id_{\underline{\M x}})$. A morphism
      $$h=(a,b):(x,f:x'\rightarrow \M x)\rightarrow (y,g:y'\rightarrow \M y)$$ is a pair consisting of $a\in \X(x,y)$ and $b\in \FL(x',y')$ making the diagram commutes.
      $$\xymatrix{
         x'\ar[r]^f \ar[d]^b & \M x \ar[d]^{\M a} \\
         y' \ar[r]^{g} & \M y
         }
      $$
      the structure morphism $\Phi\M\rightarrow\FL$ is $(x, f:x'\rightarrow \M x)\mapsto x'$, and $(a,b)\mapsto b$.
  \item for a 1-morphism $(F,F^\dag):\M_\X\rightarrow\M_\Y$, $\Phi(F,F^\dag):\Phi\M_\X\rightarrow\Phi\M_\Y$ is defined as follows:
  $$(x,f:x'\rightarrow\M_\X x)\mapsto(Fx,f\cdot F^\dag_x:x'\rightarrow \M_\Y Fx)$$
  $$\xymatrix{
  x'\ar[r]^f \ar[d]^b& \M_\X x \ar[d]^{\M_\X a} & \ar@{}[d]^{\mapsto}&x'\ar[r]^{f\cdot F^\dag_x} \ar[d]^b &\M_\Y Fx \ar[d]^{\M_\Y Fa}\\
  y'\ar[r]^g & \M_\X x                               &       & y'\ar[r]^{g\cdot F^\dag_y} &\M_\Y Fy
  }$$
  We say that $(F,F^\dag)$ is \textbf{strict} if $F^{\dag}$ is an isomorphism.
  \item for a 2-morphism $\eta:(F,F^\dag)\rightarrow(G,G^\dag)$, we define $\Phi\eta(f:x'\rightarrow\M_\X x)$ as:
  $$\xymatrix{
  Fx\ar[d]^{\eta_x} &x'\ar[r]^{f\cdot F^\dag_x} \ar[d]^{id_{x'}} &\M_\Y Fx \ar[d]^{\M_\Y \eta_x}\\
  Gx & x'\ar[r]^{f\cdot G^\dag_x} &\M_\Y Gx
  }$$
\end{enumerate}
  It's easy to see that $\Phi$ respects composition operator $\circ,\cdot,\star$ and identities. Hence $\Phi$ is a strict 2-functor. Moreover, it's routine to check that:
\begin{Proposition}\label{AB_Phi}
Assume setting \textbf{S1}. Given $\M_1\in \mathbf{FLogCat}/\C$ and $\M_2\in\mathbf{LogCat}/\C$, the canonical functor $\mathbf{LogCat}/\C(\M_1,\M_2)\rightarrow\mathbf{Cat}/\FL(\Phi\M_1,\Phi\M_2)$ is an isomorphism. $\Phi$ maps groupoid fibrations (resp. presheaves) over $\C$ with log structures to groupoid fibrations (resp. presheaves) over $\FL$. In particular, $\Phi|_{\mathbf{FLogCat}/\C}$ is strictly fully faithful.
\end{Proposition}
\textbf{Remark:} W.D.Gillam considered the special case of the notion above when $\FL=CL$ \cite{Gi1}.

We denote $\Phi_{CFG}$ the restriction of $\Phi$ on $\mathbf{LogCFG}/\C$ and call it the \emph{Gillam functor}, then $\Phi$ gives an 'embedding' of $\mathbf{FLogCFG}/\C$ into $\mathbf{CFG}/\FL$. In \cite{Gi1}, Gillam describes the essential images of $\Phi_{CFG}|_{\mathbf{FLogCFG}/\C}$, by using the notion of minimal objects:
\begin{Definition}\label{minimalob}
For a functor $F:\X\rightarrow \FL$, we say that an object $x\in \X$ is minimal if for any solid diagram in $\X$
$$\xymatrix{
w \ar@{.>}[rr]^k & & x\\
& w' \ar[ul]^i \ar[ru]_j &
}$$
with $\underline{Fi}=\underline{Fj}=id$, there is a unique completion $k$.
\end{Definition}

The following lemma is just the restatement of Gillam's result (\cite{Gi1}):
\begin{Lemma}(\cite{Gi1})\label{GC}
Under setting \textbf{S1}, the essential images of $\Phi_{CFG}|_{\mathbf{FLogCFG}/\C}$ are those $F:\X\rightarrow \FL$ in $\mathbf{CFG}/\FL$ satisfying:
\begin{description}
  \item[B1] (Enough minimal objects) For every $x\in\X$, there is a minimal object (\ref{minimalob}) $z\in\X$ and a morphism $i:x\rightarrow z$ with $\underline{Fi}=id$.
  \item[B2] (Compatibility) For any $i\in \X(w,z)$ with $z$ minimal, $Fi$ is cartesian if and only if $w$ is minimal.
\end{description}
\end{Lemma}
\textbf{Remark:} In fact $\Phi^{-1}_{CFG}(F:\X\rightarrow\FL)=(\X_m,F|_{\X_m})$ where $\X_m$ is the category of minimal objects over $\C$. For $\Phi_{CFG} F$, it can be shown that the minimal objects are exactly those of the form $(x, f)$, where $x\in \X$ and $f$ is an isomorphism.

\begin{Proposition}\label{abGC}
Assume setting \textbf{S1}, \textbf{S4}. Given $(F,\M)\in\mathbf{LogCFG}/\C$, if $F$ is a stack over $\C_\tau$, then $\Phi_{CFG} \M$ is a stack over $\FL_\tau$. If $(F,\M)\in\mathbf{FLogCFG}/\C$, then the converse is true, i.e., $F$ is a stack over $\C_\tau$ if and only if $\Phi_{CFG} \M$ is a stack over $\FL_\tau$.
\end{Proposition}

\begin{proof}
Given $F:\X\rightarrow \C$ with log structure $\M:\X\rightarrow\C^\ast$ and $\Phi_{CFG} \M:\X'\rightarrow \FL$.

  \textbf{$F$ stack imply $\Phi_{CFG} \M$ stack:}

  Given a descent data of morphisms in $\X'$ over a cover $\{c'_i\rightarrow c'\}_{i\in I}$:

  Fix $(x,f:c'\rightarrow \M x)$ and $(y,g:c'\rightarrow \M y)$ over $(c',id_{c'})$. Let $(x_i,f_i:c'_i\rightarrow \M x_i)$ (resp. $(y_i,g_i:=c'_i\rightarrow \M y_i)$) be the pullback along $c_i\rightarrow c$.
  Assume that we have compatible morphisms $h_i=(a_i,id_{c'_1}):(x_i,f_i:c'_i\rightarrow \M x_i)\rightarrow (y_i,g_i:c'_i\rightarrow \M y_i)$ with $a_i\in \X(x_i,y_i)$, making the diagrams
  $$\xymatrix{
    c'_i\ar[r]^{f_i} \ar[d]^{id_{c'_i}} & \M x_i \ar[d]^{\M a_i} \\
    c'_i \ar[r]^{g_i} & \M y_i
    }
  $$
  commute.

  Since $F$ is a stack over $\C$, and the data $(a_i)$ are compatible over $\{\underline{c'_i}\rightarrow \underline{c'}\}_{i\in I}$, $(a_i)$ is effective which glue to an (unique) object $a$. Since the pullback of $g$ and $\M af$ to $\underline{c'_i}$ are equal, we have $g=\M af$ because $h_{\M y}$ over $\FL_{\tau}$ is a sheaf.

  If we have a descent data of objects $(x_i,f_i:c'_i\rightarrow \M x_i)$ over $\{c'_i\rightarrow c'\}_{i\in I}$.
  Then $(x_i)$ are effective since $(x_i)$ is a descent data over $\C$, and $(f_i)$ are effective since $h_{\M x}$ over $\FL_{\tau}$ is a sheaf, which means that $\Phi_{CFG} \M$ is a stack.

  \textbf{$\Phi_{CFG} \M$ stack imply $F$ stack: provided $(F,\M)\in\mathbf{FLogCFG}/\C$}

  $\X$ is equivalent to $\X'_m$, the sub-fibration consists of minimal objects in $\X'$. A descent data (of objects and morphisms) in $\X'_m$ over $\{\underline{c'_i}\rightarrow \underline{c'}\}_{i\in I}$ is naturally a descent data in $\X'$ over $\{c'_i\rightarrow c'\}_{i\in I}$. Hence $\X'_m$ is a stack from the next simple lemma.
\end{proof}
\begin{Lemma}
$\X$ is a stack satisfying (\textbf{B2}) over $\FL$. $\{\iota_i:c'_i\rightarrow c'\}_{i\in I}$ is a cover. Then $x\in \X$ is minimal if and only if for each $i\in I$, the restriction $x_i=\iota_i^\ast x$ is minimal.
\end{Lemma}
\begin{proof}
The necessity follows from (\textbf{B2}). For sufficiency, assume that for each $i\in I$, $x_i=\iota_i^\ast x$ is minimal. Consider the solid diagram:
$$\xymatrix{
w \ar@{-->}[rr]^k & & x\\
& w' \ar[ul]^u \ar[ru]_v &
}$$
with $\underline{Fu}=\underline{Fv}=id$, we want the unique dashed completion $k$.

Now base change the diagram to $c'_i$ and $c'_{ij}=c'_i\times_{c'}c'_j$:
$$\xymatrix{
w_i \ar@{-->}[rr]^{k_i} & & x_i & & w_{ij} \ar@{-->}[rr]^{k_{ij}} & & x_{ij}\\
& w'_i \ar[ul]^{u_i} \ar[ru]_{v_i} & && & w'_{ij} \ar[ul]^{u_{ij}} \ar[ru]_{v_{ij}}
}$$
Since $x_i$ and $x_{ij}$ are minimal, we get a descent data $(k_i,k_{ij})$, which glue to the unique $k:w\dashrightarrow x$ that we want.
\end{proof}

Next we consider 2-fiber products in $\mathbf{LogCFG}/\C$, and $\mathbf{CFG}/\FL$:

For $\mathbf{CFG}/\FL$, we just use the 2-fiber products of groupoid fibration.

For $\mathbf{LogCFG}/\C$, if we have $(F_i,\M_i)\in\mathbf{LogCFG}/\C$, $i=1,2,3$
$$\xymatrix{
(F,\M)\ar@{-->}[r]^{P_2} \ar@{-->}[d]^{P_3} & (F_2,\M_2) \ar[d]^{T_2} \\
(F_3,\M_3) \ar[r]^{T_3} & (F_1,\M_1)
}$$

We construct the 2-fiber product $(F,\M)$ as follows:

$F=F_2\times_{F_1}F_3$ be the 2-fiber product in $\mathbf{CFG}/\C$, and $\M(x_2,x_3,\alpha:T_2x_2\simeq T_3x_3)=\M_2x_2\times_{\M_1T_2x_2}\M_3x_3$ be the fiber product of diagram
$$\xymatrix{
 & \M_3x_3 \ar[d]^{\alpha T_3^\dag(x_3)} \\
\M_2x_2 \ar[r]^{T_2^\dag(x_2)} & \M_1T_2x_2
}$$
\textbf{Remark:} one can of course use the alternate fiber product
$$\xymatrix{
 & \M_3x_3 \ar[d]^{T_3^\dag(x_3)} \\
\M_2x_2 \ar[r]^{\alpha^{-1} T_2^\dag(x_2)} & \M_1T_3x_3
}$$
which is different from the previous one by a 2-isomorphism.

It's routine to verify the follow lemma:
\begin{Lemma}\label{AB_Phi_preserveproduct}
$\Phi_{CFG}$ preserves fiber products.
\end{Lemma}

Here we would like to mention a remark on 2-fiber products in $\mathbf{FLogCFG}/\C$, which will be used in the following sections.

\begin{Definition}
An arrow $a\rightarrow b$ in $\FL$ is called integral if for any $c\rightarrow b$ in $\FL$, the fiber product $a\times_b c$ (do product in $\CL$) is in $\FL$. A morphism $(F,F^\dag):(X,\M)\rightarrow(X',\M')$ in $\mathbf{FLogCFG}/\C$ is called integral if for any $x\in X$, $F^\dag_x$ is integral.
\end{Definition}

It's direct to show that:

\begin{Lemma}\label{fiberpro_int}
Consider a 2-fiber product $\M_2\times_{\M_1}\M_3$, if $\M_1$, $\M_2$, $\M_3$ are in $\mathbf{FLogCFG}/\C$, and either $\M_2\rightarrow\M_1$ or $\M_3\rightarrow\M_1$ is integral, then $\M_2\times_{\M_1}\M_3\in\mathbf{FLogCFG}/\C$.
\end{Lemma}

\subsection{Gillam's Functor in Log Geometry}\

Now we return to log geometry. Consider $\mathbf{CLog}_S\rightarrow \mathbf{Sch}_{\underline{S}}$, with full sub-fibration $\mathbf{FLog}_S\rightarrow \mathbf{Sch}_{\underline{S}}$. Then

\begin{Proposition}\label{Setting_of_log}\
\begin{enumerate}
  \item The inclusion functor $\mathbf{FLog}_S\subseteq\mathbf{CLog}_S$ has a right adjoint, which sends strict morphisms to strict morphisms. (this is just Proposition 2.7 in \cite{KKato})
  \item $\mathbf{FLog}_S$ is a general fppf-stack over $\mathbf{Sch}_{\underline{S}}$.
  \item For any $X\in\mathbf{CLog}_S$, the presheaf of $T$-point $h_X(T)=\mathbf{CLog}_S(T,X)$, $T\in\mathbf{FLog}_S$ is a sheaf over $\mathbf{FLog}_S$
\end{enumerate}
\end{Proposition}

Hence the abstract Gillam's functor applies. Denote
$$\Phi_{log}:\mathbf{LogCFG}/{\mathbf{Sch}_{\underline{S}}}\rightarrow\mathbf{CFG}/\mathbf{Flog}_S$$
the Gillam functor in this case. This paper aim to understand how two notions of log moduli stacks are related under this functor.

Before we prove the proposition, let's state Olsson's comparison theorem of log structures.

$X$ is a scheme, with $X_{fl}$ (resp. $X_{et}$) the ringed topoi over small fppf (resp. \'etale) site. Given log structure $\M\rightarrow \mathcal{O}_{X_{et}}$, we get a prelog structure $\pi^{-1}\M\rightarrow\pi^{-1}\mathcal{O}_{X_{et}}\rightarrow\mathcal{O}_{X_{fl}}$, with the associated log structure $\pi^\ast\M\rightarrow\mathcal{O}_{X_{fl}}$. This gives a functor $\pi^\ast$ from the category of \'etale fine log structures on $X$ to the category of fppf fine log structures on $X$.

\begin{Theorem}(\cite{Ol1} Corollary A.1)\label{Olsson_descent_fine_structure}
The functor $\pi^\ast$ is an equivalence.
\end{Theorem}
\textbf{Remark:} It's easy to generalize this result to the case when $X$ is an algebraic stack.

\begin{proof}[Proof of Proposition]
(1) is just Proposition 2.7 in \cite{KKato}. That integral cover sends strict morphisms to strict morphisms is directly from the construction.

(2) We may assume that the log structures are all on fppf topology by Theorem \ref{Olsson_descent_fine_structure}. Notice that $\mathbf{FLog}_S$ is equivalent to the fibered category over the category of $\underline{S}$-schemes whose objects are morphisms
$$(\underline{X},\M_X)\rightarrow(\underline{S},\M_S)$$,
where $\M_X$ is a fine log structure on $\underline{X}_{fl}$. Since log structures and morphisms of log structures
in the fppf topology may be constructed fppf-locally, it follows that $\mathbf{FLog}_S$ is a stack with respect to the fppf topology.

(3)  We may just assume $X\in\mathbf{FLog}_S$ since $h_X=h_{X^{int}}$ by the universal property of $X^{int}$. Hence $h_X=h_{X^{int}}$ is a sheaf from (2).
\end{proof}
We give topology to $\mathbf{CLog}_S$ where the covers are strict fppf covers. Denote $\mathbf{LogStack}/\mathbf{Sch}_{\underline{S}}$ (resp. $\mathbf{FLogStack}/\mathbf{Sch}_{\underline{S}}$) the 2-category of stacks over $\mathbf{Sch}_{\underline{S}}$ with (resp. fine) log structures, and $\mathbf{Stack}^B/\mathbf{FLog}_S$ the 2-category of stacks over $\mathbf{FLog}_S$ with enough compatible minimal objects (\ref{minimalob}).
\begin{Corollary}\label{logGC}
The functor $$\Phi_{log}:\mathbf{LogCFG}/\mathbf{Sch}_{\underline{S}}\rightarrow\mathbf{CFG}/\mathbf{FLog}_S$$ restricts to 2-cartesian preserving functor
$$\Phi_{log}^{sta}:\mathbf{LogStack}/\mathbf{Sch}_{\underline{S}}\rightarrow\mathbf{Stack}/\mathbf{FLog}_S$$ and restricts to a strict equivalence of 2-categories $$\mathbf{FLogStack}/\mathbf{Sch}_{\underline{S}}\rightarrow\mathbf{Stack}^B/\mathbf{FLog}_S$$
which preserve 2-Cartesian diagram
$$\xymatrix{
\M_2\times_{\M_1}\M_3 \ar[r] \ar[d] & \M_3 \ar[d]^{F_3} \\
\M_2 \ar[r]^{F_2} & \M_1
}$$
with either $F_2$ or $F_3$ integral.
\end{Corollary}

\section{log algebraic spaces and algebraic log spaces}
\subsection{Basic notions on log versions of algebraic space}\

In this subsection we introduce various definitions of algebraic log space and their morphisms.

We fix a fine log scheme $S$, and denote $\mathbf{Sch}_{\underline{S},fppf}$ the site $\mathbf{Sch}_{\underline{S}}$ with fppf topology. We give $\mathbf{Flog}_{S}$ the topology where the covers are strict fppf covers.

Notice that if $X$ is a log scheme, then $h_X$ is a sheaf over $\mathbf{Flog}_{S,fppf}$ (Proposition \ref{Setting_of_log} (3)). We use notation $X$ instead of $h_X\simeq\Phi_{log} X$ in $\mathbf{CFG}/\mathbf{Flog}_S$.

\begin{Definition}
Let $\mathbf{P}$ be a property of morphisms between (resp. fine) log schemes. $\mathbf{P}$ is called smooth (resp. \'etale) locally on the base if: given $f:X\rightarrow Y$, and $\{U_i\rightarrow Y\}_{i\in I}$ a covering in strict log smooth (resp. \'etale) topology, then $f:X\rightarrow Y\in\mathbf{P}$ if and only if $f_{U_i}: X\times_YU_i\rightarrow U_i\in\mathbf{P}$ for any $i\in I$.
\end{Definition}

\begin{Definition}
Let $\mathbf{P}$ be a property of morphisms between (resp. fine) log schemes. $\mathbf{P}$ is called stable under base change if $f:X\rightarrow Y$ has property $\mathbf{P}$ and $U\rightarrow Y$, then the base change $f_U:X\times_YU\rightarrow U$ has property $\mathbf{P}$.
\end{Definition}
\textbf{Remark:} If $\mathbf{P}$ is a property in $\mathbf{Log}$ stable under base change, it may not restricts to a property in $\mathbf{Flog}$ stable under base change. This is because fiber products in $\mathbf{Log}$ and $\mathbf{Flog}$ are not compatible.

\begin{Definition}
If $\mathbf{P}$ is a property of morphisms in $\mathbf{Sch}$, we define the property of `strict $\mathbf{P}$' as $\mathbf{P}^{strict}=\{f\textit{ is strict and } \underline{f}\in \mathbf{P}\}$.
\end{Definition}
\textbf{Remark:} If $\mathbf{P}$ is a property of morphisms in $\mathbf{Sch}$, stable under base change (smooth or \'etale locally on the base), then $\mathbf{P}^{strict}=\{f\textit{ is strict and } \underline{f}\in \mathbf{P}\}$ is a property in $\mathbf{Log}$ ($\mathbf{Flog}$) stable under base change  (smooth or \'etale locally on the base). However, properties such as `morphism whose underlining morphism on scheme is smooth' is not stable under base change (smooth or \'etale locally on the base).

\begin{Definition}\label{DEF_ab_rep}
Let $\mathbf{P}$ (resp. $\mathbf{Q}$) be a property of morphisms in $\mathbf{Log}_S$ (resp. $\mathbf{Flog}_S$), stable under base change and smooth locally on the base.
\begin{enumerate}
  \item Let $f:\X\rightarrow\Y$ be a morphism in $\mathbf{CFG}/\mathbf{Flog}_S$. $f$ is representable if for any morphism $U\rightarrow \Y$ with $U\in\mathbf{Flog}_S$, we have $\X\times_{\Y}U\in\mathbf{Flog}_S$.  $f$ has property $\mathbf{Q}$ if for every $U\rightarrow \Y$ with $U\in\mathbf{Flog}_S$, $\X\times_{\Y}U\rightarrow U$ has property $\mathbf{Q}$.
  \item A morphism in $\mathbf{LogCFG}/\mathbf{Sch}_{\underline{S}}$ is representable if the underlining morphism in $\mathbf{CFG}/\mathbf{Sch}_{\underline{S}}$ is representable by scheme. $f$ has property $\mathbf{P}$ if for every $U\rightarrow \Y$ with $U\in\mathbf{Log}_S$, $\X\times_{\Y}U\rightarrow U$ has property $\mathbf{P}$.
\end{enumerate}
\end{Definition}
\textbf{Remark:} To check that a representable morphism in $\mathbf{LogCFG}/\mathbf{Sch}_{\underline{S}}$ has property $\mathbf{P}$, it's sufficient to check property under base change on strict morphism $U\rightarrow \Y$.
\begin{Definition}
We call a sheaf over $\mathbf{Sch}_{\underline{S},fppf}$ an algebraic space if its diagonal is representable by schemes. and admit an \'etale covering by scheme.
\end{Definition}

\begin{Definition}
$ X$ is an algebraic space with \'etale topology. A log structure of $ X$ is a pair $(\M, \alpha)$ where $\M$ is a coherent sheaf of (resp. fine) monoid and $\alpha:\M\rightarrow \mathcal{O}_{\X}$ is a homomorphism of monoids to multiply monoid of $\mathcal{O}_{ X}$, satisfying that $\alpha|_{\alpha^{-1}\mathcal{O}_{ X}^\ast}:\alpha^{-1}\mathcal{O}_{ X}^\ast\rightarrow \mathcal{O}_{ X}^\ast$ is an isomorphism. $( X,\M)$ is called an \textbf{(resp. fine) log algebraic space}. A log algebraic space is called a \emph{log scheme} if the underlining space is a scheme. $\mathbf{LAlg}_{\underline{S}}$ (resp. $\mathbf{FLAlg}_{\underline{S}}$) stands for the category of log algebraic spaces (resp. fine log algebraic spaces).
\end{Definition}
\textbf{Remark:} the category of log algebraic spaces over $\underline{S}$ is equivalent to the subcategory of $\mathbf{LogCFG}/\underline{S}$ consisting of objects whose underlining groupoid fibration are algebraic spaces.

\begin{Definition}\label{logalg}
We call a sheaf $X$ over $\mathbf{Flog}_{S,fppf}$ an \textbf{algebraic log space} if: (1) The diagonal $\Delta_X$ is representable. (2) $X$ admits a morphism $i:U\rightarrow X$ where $U$ is a fine log scheme, $i$ is strict, surjective, log \'etale. Such $i$ is called a \emph{chart} of the space. $\X\in\mathbf{CFG}/\mathbf{Flog}_{S}$ is called an algebraic log space if it is equivalent to an algebraic log space. An algebraic log space is called a \emph{fine log scheme} if it is isomorphic to a sheaf $h_X$ where $X$ is a fine log scheme. $\mathbf{AlgL}_S$ stands for the category of algebraic log spaces.
\end{Definition}
\textbf{Remark:} By an abstract argument, having representable diagram is equivalent to that every morphism $i:U\rightarrow X$ from fine log scheme is representable. Hence it make sense to say that $i$ is strict, surjective, log \'etale.

\begin{Lemma}\label{Phi_presevescheme}
If $\X\in\mathbf{LogCFG}/\mathbf{Sch}_{\underline{S}}$ is a log scheme, then $\Phi_{log} \X$ is a fine log scheme. If $\X\in\mathbf{FLogCFG}/\mathbf{Sch}_{\underline{S}}$, then $\X$ is a fine log scheme if and only if $\Phi_{log} \X$ is a fine log scheme.
\end{Lemma}
\begin{proof}
The first part is because that for log scheme $X$, $\Phi_{log}X=h_{X^{int}}$. For the second part, if $\Phi_{log} \X\simeq\Phi_{log} X$ for some fine log scheme, then $\X\simeq X$ since $\Phi_{log}|_{\mathbf{FLogCFG}/\mathbf{Sch}_{\underline{S}}}$ is strict fully faithful (Proposition \ref{AB_Phi}).
\end{proof}

The following lemma can be proved by classical argument:
\begin{Lemma}
Given $X, Y, Z\in\mathbf{CFG}/\mathbf{Flog}_{S}$ and morphisms $X\rightarrow Z$, $Y\rightarrow Z$. If $X$, $Y$, $Z$ are algebraic log spaces, then so is $X\times_ZY$.
\end{Lemma}

In section 2 we only define fiber product $X\times_SY$ when one of $X\rightarrow S$, $Y\rightarrow S$ is integral, now we construct arbitrary fiber product in $\mathbf{FLAlg}_{\underline{S}}$.

\begin{Lemma}\label{integralpart}
The inclusion functor $\mathbf{FLAlg}_{\underline{S}}\subseteq\mathbf{LAlg}_{\underline{S}}$ has a right adjoint `int', which respects strict morphism. If $f:X\rightarrow Y$ is a strict morphism, then $X\times_{Y}Y^{int}\simeq X^{int}$.
\end{Lemma}
\begin{proof}
Given a log algebraic space $(X,\M)$.

First assume that $X$ is an affine scheme $SpecA$, and $\M$ has a global chart $P\rightarrow\M$. Define $X^{int}=(SpecA\times_{Z[P]}Z[P^{int}], (P^{int})^a)$, where $P^{int}=image(P\rightarrow P^{gp})$. Then $X^{int}$ represents the sheaf $h_{(X,\M)}$ on $\mathbf{Flog}_S$, since $h_{(X,\M)}(U)$ is equal to the set of diagrams
$$\xymatrix{
Z[P] \ar[r] \ar[d] & Z[\Gamma(\underline{U},\M_U)] \ar[d] \\
A \ar[r] & \Gamma(\underline{U},\mathcal{O}_U)
}$$

which is equivalent to the set of diagrams:
$$\xymatrix{
Z[P^{int}] \ar[r] \ar[d] & Z[\Gamma(\underline{U},\M_U)] \ar[d] \\
A\times_{Z[P]}Z[P^{int}] \ar[r] & \Gamma(\underline{U},\mathcal{O}_U))
}$$
since $\M_U$ is integral.

By Yoneda lemma, the fine log scheme $(SpecA\times_{Z[P]}Z[P^{int}], (P^{int})^a)$ is independent of the choice of $P$. Hence for a general algebraic space $X$, we can choose chart \'etale locally, do the procedure above, and glue the local results to a quasi-coherent sheaf of algebra $\mathcal{O}_X^{Int}$ with a log structure $\M^{Int}$ satisfying commutative diagram:
$$\xymatrix{
\M \ar[r] \ar[d] & \mathcal{O}_X \ar[d] \\
\M^{Int} \ar[r] & \mathcal{O}_X^{Int}
}$$.
Define $X^{int}=Spec_X\mathcal{O}_X^{Int}$ with log structure induced from $\M^{Int}\rightarrow\mathcal{O}_X^{Int}$, which is functorial. One can see that any morphism from fine log algebraic space to $(X,\M)$ factors through $X^{int}$. So the functor $X\mapsto X^{int}$ is the right adjoint to the inclusion functor $\mathbf{FLAlg}_{\underline{S}}\subseteq\mathbf{LAlg}_{\underline{S}}$.

The rest of the Lemma is obvious from the construction.
\end{proof}
\textbf{Remark:} From the construction, the canonical morphism $X^{int}\rightarrow X$ is a closed immersion. We call $X^{int}$ the \emph{integral part of $(X,\M)$}.

\begin{Corollary}
$\mathbf{FLAlg}_{\underline{S}}$ has fiber products.
\end{Corollary}
\begin{proof}
The fiber product in $\mathbf{FLAlg}_{\underline{S}}$ can be constructed as follows: first take fiber products $Z$ in $\mathbf{LAlg}_{\underline{S}}$, then take the integral part of $Z$. This construction is compatible with \ref{fiberpro_int}.
\end{proof}

The next proposition is a restatement of Lemma \ref{integralpart}:

\begin{Proposition}
$\Phi_{log}|_{\mathbf{FLAlg}_{\underline{S}}}$ preserves fiber products.
\end{Proposition}

We can define properties of algebraic log space as we do to algebraic space.

\begin{Definition}
A property $\mathbf{P}$ of fine log schemes is of a local nature for the smooth (resp. \'etale) topology. If for any surjective, strict log smooth (resp. \'etale) morphism $X\rightarrow Y$, $Y\in\mathbf{P}$ if and only if $X\in\mathbf{P}$.
\end{Definition}
\textbf{Remark:} Examples are locally Noetherian, regular, normal, $S_n$, Cohen-Macaulay, reduced, of characteristic $p$, saturated, log regular (\cite{KKato2}), etc.

\begin{Definition}
Let $\mathbf{P}$ be a property of fine log schemes of a local nature for the \'etale topology. An algebraic log space $X$ has property $\mathbf{P}$ if for one (and hence for every) \'etale chart $U\rightarrow X$, $U$ has property $\mathbf{P}$. If $\mathbf{Q}$ is a property of schemes local nature for the \'etale topology, A log algebraic space has property $\mathbf{Q}$ if the underlining algebraic space has property $\mathbf{Q}$.
\end{Definition}
Hence we can say an algebraic log space: locally Noetherian, regular, normal, $S_n$, Cohen-Macaulay, reduced, of character $p$, saturated, log regular, etc.

\begin{Definition}
An algebraic log space $X$ is quasi-compact if there is a chart $U\rightarrow X$ such that $U$ is a quasi-compact fine log scheme. A morphism of algebraic log spaces $X\rightarrow Y$ is quasi-compact if for any morphism $U\rightarrow X$ from a quasi-compact fine log scheme $U$, $U\times_ Y X$ is quasi-compact. We say that $X\rightarrow Y$ is quasi-separated if the diagonal $\Delta_{X/Y}: X\rightarrow  X\times_Y X$ is quasi-compact. $X$ is called noetherian if it is quasi-compact, quasi-separate over $Spec\mathbb{Z}$, and locally noetherian.
\end{Definition}

\begin{Definition}
A property $\mathbf{P}$ of morphism in $\mathbf{Flog}$ is called smooth (resp. \'etale) local on the source-and-target if for any commutative diagram
$$\xymatrix{
X'\ar[r]^f \ar[d]^\pi & Y' \ar[d]^\varphi \\
X \ar[r]^g & Y
}$$
where $\pi$, $\varphi$ are surjective strict log smooth (resp. log \'etale) morphism, $f\in \mathbf{P}$ if and only if $g\in \mathbf{P}$.
\end{Definition}
\textbf{Remark:} The classical examples of morphisms of smooth local on the source-and-target are locally of finite representation, flat, smooth, normal, Cohen-Macaulay, $S_n$, etc (of the underling morphism). Examples of \'etale local on the source-and-target are strict \'etale, unramified.

The next result of M.Olsson is needed to define log smooth (\'etale, unramified, flat) morphisms.
\begin{Theorem}(M.Olsson \cite{Ol1})\label{Olsson_Log_X}
Let $\mathbf{Log}_S\rightarrow \mathbf{Sch}_{\underline{S}}$ be the associate groupoid fibration of $\mathbf{FLog}_S\rightarrow \mathbf{Sch}_{\underline{S}}$. Then $\mathbf{Log}_S$ is an algebraic stack locally of finite presentation $\underline{S}$, with locally separated, and finite presentation diagonal $\Delta_{\mathbf{Log}_S/\underline{S}}$. The assignment
$$S\mapsto \mathbf{Log}_S$$
is a 2-functor $$\textit{(category of fine log schemes)}\rightarrow\textit{(2-category of algebraic stacks)}$$
where $\mathbf{Log}(f)$ is always Asp-representable. Moreover, $f$ is log smooth (\'etale, unramified, flat) if and only if $\mathbf{Log}(f)$ is smooth (\'etale, unramified, flat).
\end{Theorem}

\begin{Lemma}
Log smooth, log flat, are smooth local on the source-and-target. Log \'etale, log unramified are \'etale local on the source-and-target.
\end{Lemma}
\begin{proof}
By Theorem \ref{Olsson_Log_X}, the lemma follows from that smooth, flat, representable morphisms between algebraic stack are smooth local on the source-and-target. And \'etale, unramified are \'etale local on the source-and-target. Notice that if $X\rightarrow Y$ is strict surjective, then $\mathbf{Log}_X\rightarrow\mathbf{Log}_Y$ is surjective.
\end{proof}

\begin{Definition}\label{DEF_morphisms}
Let $\mathbf{P}$ be a property of morphisms in $\mathbf{Flog}$, \'etale local on the source-and-target. A morphism $\X\rightarrow\Y$ between algebraic log spaces (fine log algebraic spaces) has property $\mathbf{P}$ if for one (and hence for every) commutative diagram
$$\xymatrix{
X\ar[r]^f \ar[d] & Y \ar[d] \\
\X \ar[r] & \Y
}$$
where the vertical arrows are strict log \'etale cover, $f$ has property $\mathbf{P}$.
\end{Definition}
Hence we can define locally of finite representation, flat, smooth, normal, Cohen-Macaulay, $S_n$, strict, integral, saturated, Kummer, Cartier, log smooth, log flat,log \'etale, log unramified morphism between algebraic log spaces.

\textbf{Remark:}
\begin{enumerate}
  \item We can also form the definition of formal log smooth (\'etale, unramified), and it turns out that log smooth (\'etale, unramified) is equivalent to locally of finite representation and formal log smooth (\'etale, unramified).
  \item If $\mathbf{P}$ is a property of morphisms smooth (\'etale) local on the source-and-target, then the associated property of morphisms of algebraic log spaces is also smooth (\'etale) local on the source-and-target.
  \item Let $\mathbf{P}$ be a property of morphisms between fine log schemes, \'etale local on the source-and-target, stable under base change and \'etale local on base. If the morphism we consider is representable (Definition \ref{DEF_ab_rep}), then the two definitions of property $\mathbf{P}$ are compatible.
\end{enumerate}
\textbf{Warning:} Suppose that $f:X\rightarrow Y$ is a morphism in $\mathbf{FLAlg}$. If $f$ is smooth under Definition \ref{DEF_morphisms}, it's not necessary that $\underline{f}$ is smooth. However, if $f$ is integral, then $\underline{f}$ is smooth. This suggests that `integral smoothness' is more natural than `smoothness'.

\subsection{Correspondence between log algebraic spaces and algebraic log spaces}\

Properties of $\Phi_{log}:\mathbf{LogCFG}/\mathbf{Sch}_{\underline{S}}\rightarrow\mathbf{CFG}/\mathbf{Flog}_S$ will be studied in this subsection. Although the topic is about sheaves, we'll prove results as general as possible.

\begin{Lemma}\label{Phi_rep_morph}
Let $\mathbf{P}$ be a property of morphism between schemes, \'etale local on the source-and-target stable under base change and smooth locally on the base. $f:\X\rightarrow \Y$ is a morphism in $\mathbf{LogCFG}/\mathbf{Sch}_{\underline{S}}$.
\begin{itemize}
  \item $\Phi_{log}$ sends a representable morphism $f:\X\rightarrow \Y$ to a representable morphism. If $f$ has property $\mathbf{P}^{strict}$, then $\Phi_{log}f$ has property $\mathbf{P}^{strict}$.

  \item If $\X\in\mathbf{FLogCFG}/\mathbf{Sch}_{\underline{S}}$, then $f$ is representable if and only if $\Phi_{log}f$ is representable.  If $f\in\mathbf{FLogCFG}/\mathbf{Sch}_{\underline{S}}$, then $f$ has property $\mathbf{P}^{strict}$ if and only if $\Phi_{log}f$ has property $\mathbf{P}^{strict}$.
\end{itemize}
\end{Lemma}
\begin{proof}
Consider the left cartesian diagram
$$\xymatrix{
V \ar[r]^{f'} \ar[d] & U \ar[d]^{u}       & \Phi_{log} V \ar[r] \ar[d] & \Phi_{log} U \ar[d]\\
\X \ar[r]^f & \Y                   & \Phi_{log} \X \ar[r]^{\Phi_{log}(f)} & \Phi_{log} \Y
}$$
where $U$ is a fine log scheme. The righthand diagram is cartesian due to Lemma \ref{AB_Phi_preserveproduct}. By Proposition \ref{AB_Phi}, $Hom(U,\Y)=Hom(\Phi_{log} U,\Phi_{log} \Y)$. If $f:\X\rightarrow \Y$ is representable, then $V$ is a log scheme (not necessarily fine). It follows that $\Phi_{log} V=h_{V^{int}}$ is a fine log scheme and $\Phi_{log}f$ is representable.

If $f$ is strict, then $V$ is fine and $f'$ is strict, hence $\Phi_{log}$ preserves $\mathbf{P}^{strict}$.

For the second part, let $\X\in\mathbf{FLogCFG}/\mathbf{Sch}_{\underline{S}}$. Assume that the diagrams we considered above are cartesian and $u$ is strict, then $V\in\mathbf{FLAlg}/\mathbf{Sch}_{\underline{S}}$. By the assumption that $\Phi_{log} f$ is representable, $\Phi_{log} V=h_{V'}$ is a fine log scheme. Hence $V\simeq V'$ is a fine log scheme because $\Phi_{log}|_{\mathbf{FLogCFG}/\mathbf{Sch}_{\underline{S}}}$ is strict fully faithful (Proposition \ref{logGC}).

If $f\in\mathbf{FLogCFG}/\mathbf{Sch}_{\underline{S}}$, then $V\rightarrow U$ is a morphism in $\mathbf{Flog}_S$. This implies that $f$ has property $\mathbf{P}^{strict}$ as long as $\Phi_{log} f$ has property $\mathbf{P}^{strict}$.
\end{proof}

\begin{Proposition}\label{Phialgsp}
Given $\X\in\mathbf{LogCFG}/\underline{S}$, if $\X$ is a log algebraic space, then $\Phi_{log}\X$ is an algebraic log space. If $\X\in\mathbf{FLogCFG}/\underline{S}$, then $\X$ is a log algebraic space if and only if $\Phi_{log}\X$ is an algebraic log space. In the latter case, for property $\mathbf{P}$ of fine log schemes local nature for the strict log \'etale topology, $\X$ has property $\mathbf{P}$ if and only if $\Phi_{log}\X$ has property $\mathbf{P}$.
\end{Proposition}
\begin{proof}
By Proposition \ref{abGC}, $\Phi_{log}\X$  is equivalent to a sheaf as long as $\X$ is (for $\X\in\mathbf{FLogCFG}/\underline{S}$, $\X$ is equivalent to a sheaf if and only if $\Phi_{log}\X$ is equivalent to a sheaf). It is sufficient to show:

\textbf{Representable of Diagonal:}

By Lemma \ref{Phi_rep_morph}, the representability of $\Delta_{\X}$ implies the representability of $\Delta_{\Phi_{log}\X}$. If $\X\in\mathbf{FLogCFG}/\underline{S}$, $\Delta_{\X}$ is representable if and only if $\Delta_{\Phi_{log}\X}$ is representable.

\textbf{Existence of Covering:}

Suppose that we have a representable strict log \'etale surjective morphism $U\rightarrow\X$, where $U$ is a log scheme. Then $\Phi_{log} U=h_{U^{int}}\rightarrow\Phi_{log}\X$ is a strict log \'etale morphism (Lemma \ref{Phi_rep_morph}).

On the other hand, assume that $\X\in\mathbf{FLogCFG}/\underline{S}$ and there is an \'etale chart $\Phi_{log} U\rightarrow\Phi_{log}\X$. By Proposition \ref{logGC}, this morphism descents to $U\rightarrow\X$. And $U\rightarrow\X$ is a strict log \'etale cover by Lemma \ref{Phi_rep_morph}.
\end{proof}

Next we study the correspondence of properties of morphisms.

\begin{Proposition}\label{Phi_premorph_Asp}
Let $\mathbf{P}$ (resp. $\mathbf{Q}$) be a property of morphisms between (resp. fine log) schemes \'etale local on the source-and-target. Then
\begin{enumerate}
  \item If a morphism $f\in Mor(\mathbf{LAlg}/\underline{S})$ has property $\mathbf{P}^{strict}$ (resp. quasi-compact, quasi-separate, representable) then $\Phi_{log} f$ has property $\mathbf{P}^{strict}$ (resp. quasi-compact, quasi-separate, representable).
  \item If $f\in Mor(\mathbf{FLAlg}/\underline{S})$, then $f$ has property $\mathbf{Q}$ (resp. quasi-compact, quasi-separate, representable) if and only if $\Phi_{log} f$ does.
\end{enumerate}
\end{Proposition}
\begin{proof}
(1) Suppose that we have diagrams
$$\xymatrix{
V \ar[r]^{f'} \ar[d]^v & U \ar[d]^{u}       & \Phi_{log} V^{int} \ar[r]^{\Phi_{log} f'} \ar[d] & \Phi_{log} U^{int} \ar[d]\\
\X \ar[r]^f & \Y                   & \Phi_{log} \X \ar[r]^{\Phi_{log}(f)} & \Phi_{log} \Y
}$$
By Lemma \ref{Phi_rep_morph}, if the left diagram is a chart of $f$ (where $u$, $v$ are strict \'etale covering by log schemes), then the right one is a chart of $\Phi_{log} f$. If $f$ has property $\mathbf{P}^{strict}$, then $f'\in \mathbf{P}^{strict}$. Since $\mathbf{P}$ stable under base change, $\Phi f'\in \mathbf{P}^{strict}$. Hence $\Phi_{log}$ preserves $\mathbf{P}^{strict}$.

Consider the case when $f$ is quasi-compact. Assume that the diagrams are cartesian and $U$ is a quasi-compact fine log scheme. then $V$ is quasi-compact. This implies that $V^{int}$ is quasi-compact and $\Phi f$ is quasi-compact.

For quasi-separateness, one notice that $\Phi_{log}\Delta_{\X/\Y}=\Delta_{\Phi_{log}\X/\Phi_{log}\Y}$.

For representability, it's Lemma \ref{Phi_rep_morph}.

(2) If $f\in Mor(\mathbf{FLAlg}/\underline{S})$, we choose an strict \'etale chart of $f$, then the righthand diagram is a chart of $\Phi_{log} f$, $U$, $V$ are fine log schemes. Hence the result holds.

Consider the case when $\Phi_{log} f$ is quasi-compact. Assume that the diagrams are cartesian, where $U$ is a quasi-compact fine log scheme and $u$ is strict, then $V$ is fine and $\Phi_{log} V$ is quasi-compact. Hence $f$ is quasi-compact.

For quasi-separateness, one notice that $\Phi_{log}\Delta_{\X/\Y}=\Delta_{\Phi_{log}\X/\Phi_{log}\Y}$. If $\Delta_{\Phi_{log}\X/\Phi_{log}\Y}$ is quasi-compact, then $\X\rightarrow (\X\times_{\Y}\X)^{int}$ is quasi-compact. Since  $(\X\times_{\Y}\X)^{int}\rightarrow\X\times_{\Y}\X$ is quasi-compact, $\Delta_{\X/\Y}$ is quasi-compact.

For representability, it's Lemma \ref{Phi_rep_morph}.
\end{proof}

\begin{Proposition}\label{Phi_preproperty_Asp}
Let $\mathbf{P}$ (resp. $\mathbf{Q}$) be a property of morphisms between (resp. fine log) schemes of a local nature for the (resp. strict log) \'etale topology.
\begin{enumerate}
  \item If $X \in \mathbf{LAlg}_{\underline{S}}$ has property $\mathbf{P}$, so is $\Phi_{log} X$.
  \item If $X\in\mathbf{FLAlg}/{\underline{S}}$, then $X$ has property $\mathbf{Q}$ if and only if $\Phi_{log} X$ property $\mathbf{Q}$.
\end{enumerate}
\end{Proposition}
\begin{proof}
This result follows from the fact that if $U\rightarrow\X$ (where $U$ is log scheme) is a chart, then $\Phi_{log} U\rightarrow\Phi_{log}\X$ is also a chart. If $\X\in\mathbf{FLAlg}/\underline{S}$, then $U\rightarrow\X$ is a chart if and only if $\Phi_{log} U\rightarrow\Phi_{log}\X$ is a chart (Lemma \ref{Phi_rep_morph}).
\end{proof}

In the end of this section we prove that an algebraic log space always has enough compatible minimal objects.

\begin{Lemma}\label{factorbystrict_Asp}
$\X$ is a stack over $\mathbf{Flog}_S$, with a representable, strict, surjective, flat and locally of finite
  presentation morphism $U\rightarrow\X$ where $U$ is a fine log scheme. $f:T\rightarrow\X$ is a morphism from a fine log scheme $T$ to $\X$ over $\mathbf{Flog}_S$. Then $f$ factors through $\xymatrix{T\ar[r]^g & T_0 \ar[r]^h & \X}$ where $T_0$ is a fine log scheme, $\underline{g}=id$, $h$ is strict (i.e. $T_0\times_{\X}U\rightarrow U$ is strict). The factorization is unique in the following sense: if there is another factorization $\xymatrix{f':T\ar[r]^{g'} & T'_0 \ar[r]^{h'} & \X}$ with a 2-isomorphism $\alpha:f\simeq f'$ where $\underline{g'}=id$, $h'$ is strict, then there is a unique pair $(u, \eta)$ where $u$ is a 1-automorphism $u:T_0\rightarrow T'_0$ with $g'=ug$, and $\beta$ is a 2-isomorphism $\beta:h\simeq h'u$ s.t. $g^\ast\beta=\alpha$. We call such factorization a strict factorization.
\end{Lemma}
\begin{proof}\
\begin{description}
  \item[Existence] Consider the solid diagram
  $$\xymatrix{
  U_T\times_TU_T \ar[r]^u \ar@<1ex>[d] \ar@<-1ex>[d] & R \ar[r] \ar@<1ex>[d] \ar@<-1ex>[d] & U\times_\X U \ar@<1ex>[d] \ar@<-1ex>[d] \\
  U_T \ar[r]^v \ar[d] & V \ar[r] \ar@{-->}[d] & U \ar[d] \\
  T \ar@{-->}[r] \ar@/_/[rr] & T_0 \ar@{-->}[r]^h & \X
  }$$
  where the left vertical arrows come from base change of the right vertical arrows. The first and second horizontal arrows are the strict factorization of log schemes. Since $\underline{u}=id$, $\underline{v}=id$, we obtain that $\xymatrix{\underline{R} \ar@<1ex>[r] \ar@<-1ex>[r] &\underline{V}}$ is effective with quotient $id:\underline{T}\rightarrow \underline{T_0}$. Moreover, since $\underline{V}\rightarrow \underline{T_0}$ is flat, locally of finite presentation morphism and $\xymatrix{R \ar@<1ex>[r] \ar@<-1ex>[r] &V}$ are strict, we can descent the log structure on $V$ to $\underline{T_0}$ (Theorem \ref{Olsson_descent_fine_structure}). Denote this decent fine log scheme $T_0$, then $V\rightarrow T_0$ is strict, flat, locally of finite presentation. Since $\X$ is a stack, the descent data of morphism $\xymatrix{R \ar@<1ex>[r] \ar@<-1ex>[r] &V \ar[r] & \X}$ gives $h:T_0\rightarrow \X$ fitting in the diagram.

  To show $V\simeq T_0\times_{\X}U$ in the diagram, which implies that $h$ is strict, we consider the diagram:
  $$\xymatrix{
   & V \ar[d]^{i} \ar[rd] & \\
  U_T \ar[r]^{f} \ar[ru]^g & T_0\times_{\X}U \ar[r] & U
  }$$
  where $\underline{f}=\underline{g}=id$. Since $V\rightarrow U$ is strict, $i$ is an isomorphism, hence $V\simeq T_0\times_{\X}U$.

  This gives a strict factorization of $T\rightarrow\X$.

  \item[Uniqueness] Using the same diagram of another strict factorization:
  $$\xymatrix{
  U_T\times_TU_T \ar[r]^u \ar@<1ex>[d] \ar@<-1ex>[d] & R' \ar[r] \ar@<1ex>[d] \ar@<-1ex>[d] & U\times_\X U \ar@<1ex>[d] \ar@<-1ex>[d] \\
  U_T \ar[r]^v \ar[d] & V' \ar[r] \ar[d] & U \ar[d] \\
  T \ar[r]^{g'} \ar@/_/[rr]_{f'} & T'_0 \ar[r]^{h'} & \X
  }$$
  with a 2-isomorphism $\alpha:f\simeq f'$, $R'$ and $V'$ come from the pullback through $T'_0\rightarrow\X$. Then $R'$ and $V'$ give a strict factorizations. By the uniqueness of the strict factorization of log schemes, there are unique isomorphisms $r:R\simeq R'$, $v:V\simeq V'$ compatible to the diagrams. We can descent them to an isomorphism $u:T_0\simeq T'_0$ compatible to the diagrams. Doing a same descent procedure, we get a 2-isomorphism $\beta:h\simeq h'u$, s.t. $g'^\ast\beta=\alpha$. The uniqueness of $u$ and $\beta$ comes from chasing the diagram.
\end{description}
\end{proof}

\begin{Theorem}\label{alglogstackminimalweak}
Suppose that $\X$ is a stack over $\mathbf{Flog}_S$, with a representable, strict, surjective, flat and locally of finite
  presentation morphism $U\rightarrow\X$ where $U$ a fine log scheme. Then $\X$ has enough compatible minimal objects.
\end{Theorem}
\begin{proof}
Let $\X_m$ be the subcategory of $\X$ consists of objects corresponding to strict morphisms $T\rightarrow X$ (that is to say $T\times_{\X}U\rightarrow U$ is strict) where $T$ is a fine log scheme. We prove that $\X_m$ form a compatible system of minimal objects.
\begin{description}
  \item[Step 1: Minimality] Given a 2-commutative solid diagram
  $$\xymatrix{
  T_1 \ar[r]^u \ar[d]^v \ar[rd]^{\xi_1} & T_2 \ar[d]^{\xi_2} \ar@{-->}[r]^g & T'_2 \ar@{-->}[ld]^h \ar@{-->}[lld]_w\\
  T_0 \ar[r]_{\xi_0} & \X &
  }$$
  with 2-isomorphisms $\alpha_1:\xi_1\simeq \xi_0 v$, $\alpha_2:\xi_2 u\simeq \xi_1$, and $\xi_0$ is strict, $\underline{u}=\underline{v}=id$. Take a strict factorization $\xi_2=hg$. It follows that $\xi_1$ has two strict factorizations $\alpha_1:\xi_1\simeq\xi_0v=h(gu)$. By the uniqueness of the strict factorization we have a unique 1-automorphism $w:T'_2\rightarrow T_0$ s.t. $wgu=v$, and a unique 2-isomorphism $\beta: h\simeq \xi_0 w$ s.t. $(gu)^\ast(\beta)=\alpha_1\alpha_2$. Hence $wg$ is what we need.

  To prove the uniqueness, assume that there are morphisms $\phi_i:T_2\rightarrow T_0$, with 2-isomorphisms $\beta'_i: \xi_2 \simeq \xi_0 \phi_i$ s.t. $u^\ast(\beta'_i)=\alpha_1\alpha_2$ ($i=1, 2$). This gives two strict factorizations of $\xi_2$. By the uniqueness, we have $\phi_1=\phi_2$ and $\beta_2\beta_1^{-1}=id$. This finishes the proof that $\X_m$ are all minimal objects.

  \item[Step 2: $\X_m$ is enough] This is the direct consequence of the existence of the strict factorization by Lemma \ref{factorbystrict_Asp}.
  \item[Step 3: Compatibility] Let $\xi:T\rightarrow X \in\X_m$. Given $f:T_1\rightarrow T$, we have $f^\ast\xi:T_1\rightarrow T\rightarrow\X$. Hence $f^\ast\xi$ is strict if and only if $f$ is strict by chasing the diagram
      $$\xymatrix{
      U_{T_1} \ar[r]^{f_U} \ar[d] & U_{T_2} \ar[r] \ar[d] & U \ar[d] \\
      T_1 \ar[r]^f & T_2 \ar[r] & \X
      }$$.
      Notice that $f$ is strict if and only if $f_U$ is strict.

      This proves the compatibility of $\X_m$.

\end{description}
\end{proof}

As a corollary, all algebraic log spaces have enough compatible minimal objects. Hence we get the representation theorem of algebraic log space:
\begin{Theorem}\label{Phi_corres_Asp}
$\Phi_{log}$ sends $\mathbf{LAlg}_{\underline{S}}$ to $\mathbf{AlgL}_S$, and restricts on $\mathbf{FLAlg}_{\underline{S}}$ to be a strict 2-equivalence:
$$\mathbf{FLAlg}_{\underline{S}}\rightarrow\mathbf{AlgL}_S$$
Moreover, $\Phi_{log} X\simeq\Phi_{log} Y$ if and only if $X^{int}\simeq Y^{int}$.
\end{Theorem}
\begin{proof}
By Proposition \ref{Phialgsp}, $\Phi_{log}$ sends $\mathbf{LAlg}_{\underline{S}}$ to $\mathbf{AlgL}_S$, and restricts to a strict fully faithful functor $\mathbf{FLAlg}_{\underline{S}}\rightarrow\mathbf{AlgL}_S$. By Theorem \ref{alglogstackminimalweak} and Corollary \ref{logGC}, the essential images of $\Phi_{log}$ are $\mathbf{AlgL}_S$. Hence $$\mathbf{FLAlg}_{\underline{S}}\rightarrow\mathbf{AlgL}_S$$ is a strict 2-equivalence.
The last part follows directly from Lemma \ref{integralpart}.
\end{proof}

\begin{Definition}\label{DEF_underlinespace-AlgL}
Given an algebraic log space $X$, there is unique (up to equivalence) a fine log algebraic space $X'$ s.t. $\Phi_{log} X'\simeq X$. The underlining algebraic space of $X$ is defined to be $\underline{X'}$, and the log structure of $X$ is defined to be $\M_{X'}$.
\end{Definition}

As a trivial corollary, every morphism $f:X\rightarrow Y$ between algebraic log spaces has a strict factorization as Lemma \ref{factorbystrict_Asp}.
\begin{Lemma}\label{factorbystrict_Aspace}
$f:Y\rightarrow X$ is a morphism in $\mathbf{AlgL}_S$. Then $f$ factors through $\xymatrix{Y\ar[r]^g & Y_0 \ar[r]^h & \X}$ where $Y_0$ is an algebraic log space, $\underline{g}=id:\underline{Y}\rightarrow\underline{X}$ and $h$ is strict. The factorization is unique in the following sense: if there is another factorization $\xymatrix{f':Y\ar[r]^{g'} & Y'_0 \ar[r]^{h'} & \X}$ where $\underline{g'}=id$ and $h'$ is strict, then there is a unique automorphism $u:T_0\rightarrow T'_0$ with $g'=ug$. We call such factorization a strict factorization.
\end{Lemma}
\begin{proof}
By theorem \ref{Phi_corres_Asp}, $f=\Phi_{log}f'$ comes from $f':X'\rightarrow Y'$ in $\mathbf{FLAlg}_{\underline{S}}$. We can take the strict factorization $f':X'\rightarrow Z'\rightarrow Y'$ in $\mathbf{FLAlg}_{\underline{S}}$, then $f:X\rightarrow \Phi_{log} Z'\rightarrow Y$ is the strict factorization that we want.
\end{proof}

We make a remark about descent of morphism from algebraic log spaces to stack over $\mathbf{Flog}_S$. We give $\mathbf{FLAlg}_S$ the big strict fppf topology, where coverings $\{f_i:X_i\rightarrow X\}_{i\in I}$ are strict fppf morphisms s.t. $\coprod_{i\in I} f_i:\coprod_{i\in I}X_i\rightarrow X$ is surjective. We denote this topoi $\mathbf{FLAlg}_{S,fppf}$.

Given a stack $\X$ over $\mathbf{Flog}_S$, we can define the groupoid fibration $\X_{Asp}$ over $\mathbf{FLAlg}_{S,fppf}$. The objects of $\X_{Asp}$ are morphisms $f:U\rightarrow \X$ in $\mathbf{CFG}_{\mathbf{Flog}_S}$ with $U\in \mathbf{FLAlg}$. A morphism $\alpha:f\rightarrow g$ is a natural transformation from $f$ to $g$ over $\mathbf{Flog}_S$ (i.e. a 2-morphism from $f$ to $g$ in the strict 2-category $\mathbf{CFG}_{\mathbf{Flog}_S}$). The following descent result will be needed in Theorem \ref{factorbystrict_Stack}.
\begin{Lemma}\label{descent_Asptostack}
$\X_{Asp}$ is a stack over $\mathbf{FLAlg}_{S,fppf}$.
\end{Lemma}

\section{Log Algebraic Stacks and Algebraic Log Stacks}

\subsection{Basic notions on log versions of algebraic stack}\

In this section we continue to study log version of algebraic stacks, using the same setting as that in the previous section.
Most of the arguments are almost formal copies of corresponding arguments for algebraic log spaces. However, since they rely heavily on the previous section, we would like to show the complete proofs.

\begin{Definition}\label{stack}
A stack over $\mathbf{Sch}_{\underline{S},fppf}$ is algebraic (resp. DM) if its diagonal is representable by algebraic spaces and admits a smooth (resp. \'etale) covering by scheme.
\end{Definition}

\begin{Definition}\label{DEF_logalgstack}
Given an algebraic stack $\X$ with fppf topology. A (resp. fine) log structure of $\X$ is a pair $(\M, \alpha)$ where $\M$ is a coherent (resp. fine) sheaf of monoid and $\alpha:\M\rightarrow \mathcal{O}_{\X}$ is a homomorphism of monoids to multiply monoid of $\mathcal{O}_{\X}$, satisfying that $\alpha|_{\alpha^{-1}\mathcal{O}_{\X}^\ast}:\alpha^{-1}\mathcal{O}_{\X}^\ast\rightarrow \mathcal{O}_{\X}^\ast$ is an isomorphism. Such algebraic stack with logarithmic structure is called a \textbf{log algebraic stack}. $\mathbf{LAS}_{\underline{S}}$ (resp. $\mathbf{FLAS}_{\underline{S}}$) stands for the strict 2-category of (resp. fine) log algebraic stacks.
\end{Definition}
\textbf{Remark:} We can also define the log structure on the lisse-\'etale topology, which turns out to be equivalent to our notion when the log structure is fine (Theorem \ref{Olsson_descent_fine_structure}).

\begin{Definition}\label{DEF_ab_rep_Asp}
Let $\mathbf{P}$ (resp. $\mathbf{Q}$) be a property of morphisms in $\mathbf{LAlg}_S$ (resp. $\mathbf{FAlgL}_S$), stable under base change and smooth locally on the base.
\begin{enumerate}
  \item Let $f:\X\rightarrow\Y$ be a morphism in $\mathbf{CFG}/\mathbf{Flog}_S$. $f$ is \textbf{Asp-representable} if for any morphism $U\rightarrow \Y$ with $U\in\mathbf{Flog}_S$, $\X\times_{\Y}U\in\mathbf{AlgL}_S$. $f$ has property $\mathbf{Q}$ if for every $U\rightarrow \Y$ with $U\in\mathbf{Flog}_S$, $\X\times_{\Y}U\rightarrow U$ has property $\mathbf{Q}$.
  \item A morphism in $\mathbf{LogCFG}/\mathbf{Sch}_{\underline{S}}$ is Asp-representable if the underlining morphism in $\mathbf{CFG}/\mathbf{Sch}_{\underline{S}}$ is representable by algebraic spaces. $f$ has property $\mathbf{P}$ if for every $U\rightarrow \Y$ with $U\in\mathbf{Log}_S$, $\X\times_{\Y}U\rightarrow U$ has property $\mathbf{P}$.
\end{enumerate}
\end{Definition}
\textbf{Remark:} To check that an Asp-representable morphism in $\mathbf{LogCFG}/\mathbf{Sch}_{\underline{S}}$ has property $\mathbf{P}$, it's sufficient to check property under base change on strict morphism $U\rightarrow \Y$.

It is known that if $\X\rightarrow\Y$ is an Asp-representable morphism from a stack to an algebraic space, then $\X$ is an algebraic space. The log version is also true:

\begin{Lemma}\label{Asprep}
If $\X\rightarrow X$ in $\mathbf{CFG}/\mathbf{Flog}_S$ is an Asp-representable morphism from a stack to an algebraic log space. Then $\X$ is an algebraic space.
\end{Lemma}
\begin{proof}
It's sufficient to verify:
\begin{description}
  \item[$\X$ is setoid] It's sufficent to show that for any $x\in \X_{U}$, $Aut(x)=\{id_{x}\}$. Consider the cartesian diagram:
      $$\xymatrix{
      \X\times_{X}U \ar[r] \ar[d] & \X \ar[d]^{f}\\
      U \ar[r]^g & X
      }$$
      We know from the assumption that $\X\times_{X}U$ is a sheaf, i.e., for an object $\xi=(x,u,\alpha:f(x)\simeq g(u))$, $Aut(\xi)=\{id_{\xi}\}$. Pick $\beta\in Aut(x)$, since $X$ is an algebraic log space, $f(\beta)=id_{f(x)}$, hence $g(id_u)\alpha=\alpha f(\beta)$. So $(\beta, id_u)\in Aut(\xi)$, which implies that $\beta=id_x$ and $\X$ is setoid.
  \item[$\X$ has representable diagonal] Let $U\rightarrow\X\times_S\X$ be a morphism from a fine log scheme $U$, then $U\times_{\X\times_S\X}\X\simeq (U\times_X\X)\times_{(U\times_X\X)\times_U(U\times_X\X)}U$ is a fine log scheme. Hence the diagonal $\Delta_{\X}$ is representable.
  \item[$\X$ has a strict log \'etale cover] Pick a strict log \'etale cover $U\rightarrow X$ by a fine log scheme $U$, then $U\times_X\X\rightarrow\X$ is a strict surjective log \'etale morphism where $U\times_X\X$ is an algebraic log space. Choose a strict log \'etale cover $U'\rightarrow U\times_X\X$, the composition $U'\rightarrow\X$ gives a strict log \'etale cover.
\end{description}
\end{proof}

\begin{Corollary}\label{Asprep_2}
If $\X\rightarrow\Y$ in $\mathbf{CFG}/\mathbf{Fog}_S$ is Asp-representable, and $U$ is an algebraic log space with a morphism $U\rightarrow\Y$, then $\X\times_{\Y}U$ is an algebraic log space.
\end{Corollary}
\begin{proof}
Notice that $\X\times_{\Y}U\rightarrow U$ is Asp-representable, so $\X\times_{\Y}U$ is an algebraic log space by Lemma \ref{Asprep}.
\end{proof}

\begin{Definition}\label{DEF_Alglogstack}
A stack over $\mathbf{Flog}_{S,fppf}$ is an \textbf{algebraic (resp. DM) log stack} if its diagonal is Asp-representable, and admit a strict, surjective, log smooth (resp. log \'etale) morphism $i:U\rightarrow X$ where $U$ is a fine log scheme. $i:U\rightarrow X$ is called a \textbf{smooth (resp. \'etale) chart}. Denote $\mathbf{ALS}_S$ the strict 2-category of algebraic log stacks.
\end{Definition}
\textbf{Remark:} the algebraic log stack we define is general than that of Olsson's in \cite{Ol4}, by dropping the locally of finite presentable condition.

\begin{Lemma}
Given $\X, \Y, \mathcal{Z} \in \mathbf{CFG}/\mathbf{Flog}_S$ and morphisms $\X\rightarrow \mathcal{Z}$, $\Y\rightarrow \mathcal{Z}$. If $\X$, $\Y$, $\mathcal{Z}$ are algebraic log stacks, so is $\X\times_{\mathcal{Z}}\Y$.
\end{Lemma}
\begin{proof}
By standard argument.
\end{proof}
We can construct 2-fiber products in $\mathbf{FLAS}_{\underline{S}}$ as we did for algebraic log spaces.

\begin{Lemma}\label{integralpart_stack}
The inclusion functor $\mathbf{FLAS}_{\underline{S}}\subseteq\mathbf{LAS}_{\underline{S}}$ has a right adjoint `int', which respects strict morphism.  If $f:\X\rightarrow \Y$ is a strict morphism, then $\X\times_{\Y}\Y^{int}\simeq \X^{int}$.
\end{Lemma}
\begin{proof}
Given a log algebraic stack $(\X,\M)$. We construct a sheaf of log algebra over $(\mathcal{O}_{\X},\M)$ as in Lemma \ref{integralpart}:
$$\xymatrix{
\M \ar[r] \ar[d] & \mathcal{O}_{\X} \ar[d] \\
\M^{Int} \ar[r] & \mathcal{O}_{\X}^{Int}
}$$
Define $\X^{int}=Spec_{\X}\mathcal{O}_{\X}^{Int}$ with log structure induced from $\M^{Int}\rightarrow\mathcal{O}_{\X}^{Int}$, which is functorial. Any morphism from a fine log algebraic stack to $(\X,\M)$ factors through $\X^{int}$. So the functor $\X\mapsto {\X}^{int}$ is the right adjoint to the inclusion functor $\mathbf{FLAS}_{\underline{S}}\subseteq\mathbf{LAS}_{\underline{S}}$.

The rest of the Lemma is obvious from the construction.
\end{proof}
\textbf{Remark:} From the construction we see that the canonical morphism $\X^{int}\rightarrow \X$ is a closed immersion. We call $\X^{int}$ the \emph{integral part of $(\X,\M)$}.

\begin{Corollary}
$\mathbf{FLAS}_{\underline{S}}$ has fiber products.
\end{Corollary}

The next Proposition is just a restatement of Lemma \ref{integralpart_stack}:

\begin{Proposition}
$\Phi_{log}|_{\mathbf{FLAS}_{\underline{S}}}$ preserves fiber products.
\end{Proposition}

We can define properties of algebraic log stack as we do to algebraic stack.

\begin{Definition}
Let $\mathbf{P}$ be a property of fine log schemes of a local nature for the strict log smooth topology. An algebraic log stack (fine log algebraic stack) $\X$ has property $\mathbf{P}$ if one (and hence for every) of its fine smooth charts $U\rightarrow\X$, $U$ has property $\mathbf{P}$. If $\mathbf{Q}$ is a property of schemes local nature for the smooth topology, a log algebraic stack has property $\mathbf{Q}$ if the underlining algebraic stack has property $\mathbf{Q}$.
\end{Definition}
Hence we can call an algebraic log stack (fine log algebraic stack): locally Noetherian, regular, normal, $S_n$, Cohen-Macaulay, reduced, of character $p$, saturated, log regular, etc.

\begin{Definition}
An algebraic log stack is quasi-compact if there is a chart $U\rightarrow X$ such that $U$ is a quasi-compact fine log scheme. A morphism of algebraic log stacks $\X\rightarrow\Y$ is quasi-compact if for any quasi-compact fine log scheme $U$, $U\times_\Y\X$ is quasi-compact. We say that $\X\rightarrow\Y$ is quasi-separated if the diagonal $\Delta_{\X/\Y}:\X\rightarrow \X\times_\Y\X$ is quasi-compact and quasi-separated. $\X$ is called noetherian if it is quasi-compact, quasi-separate, and locally noetherian.
\end{Definition}

\begin{Definition}\label{DEF_morphisms_stack}
Let $\mathbf{P}$ be a property of morphisms in $\mathbf{Flog}$, smooth (resp. \'etale) local on the source-and-target. A morphism $\X\rightarrow\Y$ of algebraic (resp. DM) log stack (fine log algebraic (resp. DM) stack) has property $\mathbf{P}$ if for one (and hence for every) commutative diagram
$$\xymatrix{
X\ar[r]^f \ar[d] & Y \ar[d] \\
\X \ar[r] & \Y
}$$
where the vertical arrows are smooth (resp. \'etale) charts, $f$ has property $\mathbf{P}$.
\end{Definition}
Hence we can define locally of finite representation, flat, smooth, normal, Cohen-Macaulay, $S_n$, strict, integral, saturated, Kummer, Cartier, log smooth, log flat, morphisms between algebraic log stacks (log algebraic stacks). And log \'etale, log unramified morphisms between DM log stacks (More generally relatively DM-morphisms (\ref{defrep_stack})).

\textbf{Remark:} There are also notions of formal log smooth (\'etale, unramified). It turns out that log smooth is equivalent to locally of finite representation and formal log smooth. For relatively DM-morphism, log \'etale (unramified) is equivalent to locally of finite representation, and formal log \'etale (unramified).

\begin{Definition}\label{defrep_stack}
A morphism $f:\X\rightarrow\Y$ in $\mathbf{ALS}_S$ is called DM (resp. representable, Asp-representable) if for any morphism $U\rightarrow \Y$ with $U$ a fine log scheme, $\X\times_{\Y}U$ is a DM-log stack (resp. fine log scheme, algebraic log space). For a property $\mathbf{P}$ of \'etale local on the source-and-target, we say that a DM-morphism $f:\X\rightarrow\Y$ have property $\mathbf{P}$ if for one (and hence for every) strict log smooth cover $U\rightarrow \Y$, $\X\times_{\Y}U\rightarrow U$ has property $\mathbf{P}$.
\end{Definition}
\textbf{Remark:}
\begin{enumerate}
  \item Let $\mathbf{P}$ be a property of morphisms between fine log schemes, smooth local on the source-and-target, stable under base change and smooth local on base. If the morphism we consider is Asp-representable (Definition \ref{DEF_ab_rep_Asp}), then the two definitions of property $\mathbf{P}$ are compatible.
  \item If $f:\X\rightarrow\Y$ is DM and $\mathcal{Z}\rightarrow \Y$ is a morphism from DM-stack $\mathcal{Z}$, then $\X\times_{\Y}\mathcal{Z}$ is a DM-stack. 
\end{enumerate}
Hence we can define a relative DM morphisms to be log \'etale, log unramified.

\subsection{Correspondence between log algebraic stacks and algebraic log stacks}

\begin{Definition}
$\mathbf{P}$ is a property of morphisms between algebraic spaces, define the property in $\mathbf{LAlg}$ of `strict $\mathbf{P}$' as $\mathbf{P}^{strict}=\{f\textit{ is strict and } \underline{f}\in \mathbf{P}\}$.
\end{Definition}
\textbf{Remark:} If $\mathbf{P}$ is a property of morphisms in $\mathbf{Sch}$, stable under base change (smooth or \'etale locally on the base), then $\mathbf{P}^{strict}=\{f\textit{ is strict and } \underline{f}\in \mathbf{P}\}$ is a property in $\mathbf{LAlg}$ ($\mathbf{FLAlg}$) stable under base change  (smooth or \'etale locally on the base). However, properties such as `morphism whose underlining morphism on scheme is smooth' is not stable under base change (smooth or \'etale locally on the base). By Theorem \ref{Phi_corres_Asp}, we can also consider $\mathbf{P}^{strict}$ as a property in $\mathbf{AlgL}$.

\begin{Lemma}\label{Phi_Stackrep_morph}
Let $\mathbf{P}$ be a property of morphisms between algebraic spaces, \'etale local on the source-and-target stable under base change and smooth locally on the base. $f:\X\rightarrow \Y$ is a morphism in $\mathbf{LogCFG}/\mathbf{Sch}_{\underline{S}}$.
\begin{itemize}
  \item $\Phi_{log}$ sends  Asp-representable morphism $f:\X\rightarrow \Y$ to Asp-representable morphism. If $f$ has property $\mathbf{P}^{strict}$, then $\Phi_{log}f$ has property $\mathbf{P}^{strict}$.

  \item If $\X\in\mathbf{FLogCFG}/\mathbf{Sch}_{\underline{S}}$, then $f$ is Asp-representable if and only if $\Phi_{log}f$ is Asp-representable.  If $f\in\mathbf{FLogCFG}/\mathbf{Sch}_{\underline{S}}$, then $f$ has property $\mathbf{P}^{strict}$ if and only if $\Phi_{log}f$ has property $\mathbf{P}^{strict}$.
\end{itemize}
\end{Lemma}

\begin{proof}
Consider the left cartesian diagram
$$\xymatrix{
V \ar[r]^{f'} \ar[d] & U \ar[d]^{u}       & \Phi_{log} V \ar[r] \ar[d] & \Phi_{log} U \ar[d]\\
\X \ar[r]^f & \Y                   & \Phi_{log} \X \ar[r]^{\Phi_{log}(f)} & \Phi_{log} \Y
}$$
where $U$ is a fine log scheme. The righthand diagram is cartesian due to Lemma \ref{AB_Phi_preserveproduct}. By Proposition \ref{AB_Phi}, $Hom(U,\Y)=Hom(\Phi_{log} U,\Phi_{log} \Y)$.  If $f:\X\rightarrow \Y$ is Asp-representable, then $V$ is a log algebraic space (not necessarily fine). It follows that $\Phi_{log} V$ is an algebraic log space (Theorem \ref{Phialgsp}) and $\Phi_{log}f$ is representable.

Notice that if $f$ is strict, then $V$ is fine and $f'$ is strict, hence $\Phi_{log}$ preserves $\mathbf{P}^{strict}$.

For the second part, assume that $\X\in\mathbf{FLogCFG}/\mathbf{Sch}_{\underline{S}}$ and the diagrams we considered above are cartesian and $u$ is strict. Then $V\in\mathbf{FLogCFG}/\mathbf{Sch}_{\underline{S}}$. By assumption that $\Phi_{log} f$ is Asp-representable, $\Phi_{log} V=h_{V'}$ is an algebraic log space. By Proposition \ref{Phialgsp}, $V$ is a fine log algebraic space. The fact that $\Phi_{log}$ respect the property $\mathbf{P}^{strict}$ is obvious.
\end{proof}

\begin{Proposition}\label{Phialgstack}
Given $\X\in\mathbf{LogCFG}/\mathbf{Sch}_{\underline{S}}$, if $\X$ is a log algebraic stack, then $\Phi_{log}\X$ is an algebraic log stack. If $\X\in\mathbf{FLogCFG}/\mathbf{Sch}_{\underline{S}}$, then $\X$ is a log algebraic stack if and only if $\Phi_{log}\X$ is an algebraic log structure. In the latter case, for a property $\mathbf{P}$ of fine log schemes of a local nature for the strict log smooth topology, $\X$ has property $\mathbf{P}$ if and only if $\Phi_{log}\X$ has property $\mathbf{P}$.
\end{Proposition}
\begin{proof}
By Proposition \ref{abGC}, $\Phi_{log}\X$ is a stack as long as $\X$ is a stack (for $\X\in\mathbf{FLogCFG}/\mathbf{Sch}_{\underline{S}}$, $\X$ is a stack if and only if $\Phi_{log}\X$ is a stack). It is sufficient to show:

\textbf{Representable of Diagonal:}

By Lemma \ref{Phi_Stackrep_morph}, the Asp-representability of $\Delta_{\X}$ implies the Asp-representability of $\Delta_{\Phi_{log}\X}$. And when $\X\in\mathbf{FLogCFG}/\mathbf{Sch}_{\underline{S}}$, $\Delta_{\X}$ is Asp-representable if and only if $\Delta_{\Phi_{log}\X}$ is representable.

\textbf{Existence of Covering:}

Suppose that we have a representable strict log smooth surjective morphism $U\rightarrow\X$, where $U$ is a log scheme. Then $\Phi_{log} U=h_{U^{int}}\rightarrow\Phi_{log}\X$ is a smooth chart (Lemma \ref{Phi_Stackrep_morph}).

On the other hand, if $\X\in\mathbf{FLogCFG}/\mathbf{Sch}_{\underline{S}}$ and there is a smooth chart $\Phi_{log} U\rightarrow\Phi_{log}\X$. By Proposition \ref{logGC}, this morphism descents to $U\rightarrow\X$. And $U\rightarrow\X$ is a strict log smooth cover by Lemma \ref{Phi_Stackrep_morph}.
\end{proof}

Next we study the correspondence of properties of morphisms.

\begin{Proposition}\label{Phi_premorph_Astack}
Let $\mathbf{P}$ (resp. $\mathbf{Q}$) be a property of morphisms between (resp. fine log) schemes \'etale local on the source-and-target. Then
\begin{enumerate}
  \item If a morphism $f\in Mor(\mathbf{LAS}/\underline{S})$ has property $\mathbf{P}^{strict}$ (resp. quasi-compact, quasi-separate, representable), then $\Phi_{log} f$ has property $\mathbf{P}^{strict}$ (resp. quasi-compact, quasi-separate, representable).
  \item If $f\in Mor(\mathbf{FLAS}/\underline{S})$ then $f$ has property $\mathbf{Q}$ (resp. quasi-compact, quasi-separate, representable) if and only if $\Phi_{log} f$ has property $\mathbf{Q}$.
\end{enumerate}
\end{Proposition}
\begin{proof}
\
(1) Suppose that we have diagrams
$$\xymatrix{
V \ar[r]^{f'} \ar[d]^v & U \ar[d]^{u}       & \Phi_{log} V \ar[r]^{\Phi_{log} f'} \ar[d] & \Phi_{log} U \ar[d]\\
\X \ar[r]^f & \Y                   & \Phi_{log} \X \ar[r]^{\Phi_{log}f} & \Phi_{log} \Y
}$$
By Lemma \ref{Phi_Stackrep_morph}, if the left diagram is a chart of $f$ (where $u$, $v$ are smooth charts), then the right one is a chart of $\Phi_{log} f$. If $f$ has property $\mathbf{P}^{strict}$, then $f'\in \mathbf{P}^{strict}$. Since $\mathbf{P}$ is stable under base change, $\Phi_{log} f'\in \mathbf{P}^{strict}$. Hence $\Phi_{log}$ preserves $\mathbf{P}^{strict}$.

Consider the case when $f$ is quasi-compact. Assume that the diagrams are cartesian and $U$ is a quasi-compact fine log scheme. then $V$ is quasi-compact. This implies that $V^{int}$ is quasi-compact and $\Phi f$ is quasi-compact.

For quasi-separateness, one notice that $\Phi_{log}\Delta_{\X/\Y}=\Delta_{\Phi_{log}\X/\Phi_{log}\Y}$ and $\Phi_{log}\Delta_{\Delta_{\X/\Y}}=\Delta_{\Delta_{\Phi_{log}\X/\Phi_{log}\Y}}$.

For representability, it's Lemma \ref{Phi_Stackrep_morph}.

(2) If $f\in Mor(\mathbf{FLAS}/\underline{S})$, we choose an strict log smooth chart of $f$. Then the righthand diagram is a chart of $\Phi_{log} f$ and $U$, $V$ are fine log schemes. Hence the result holds.

Consider the case when $\Phi_{log} f$ is quasi-compact. Assume that the diagrams are cartesian, where $U$ is a quasi-compact fine log scheme and $u$ is strict, then $V$ is fine and $\Phi_{log} V$ is quasi-compact. Hence $f$ is quasi-compact.

For quasi-separateness, one notice that $\Phi_{log}\Delta_{\X/\Y}=\Delta_{\Phi_{log}\X/\Phi_{log}\Y}$. If $\Delta_{\Phi_{log}\X/\Phi_{log}\Y}$ is quasi-compact, then $\X\rightarrow (\X\times_{\Y}\X)^{int}$ is quasi-compact. Since $(\X\times_{\Y}\X)^{int}\rightarrow\X\times_{\Y}\X$ is quasi-compact, $\Delta_{\X/\Y}$ is quasi-compact. The same argument shows that if $\Delta_{\Delta_{\Phi_{log}\X/\Phi_{log}\Y}}$ is quasi-separated, so is $\Delta_{\Delta_{\X/\Y}}$.

For representability, it's Lemma \ref{Phi_Stackrep_morph}.
\end{proof}
\begin{Proposition}\label{Phi_preproperty_Astack}
Let $\mathbf{P}$ (resp. $\mathbf{Q}$) be a property of (resp. fine log) schemes of a local nature for the (resp. strict log) smooth topology.
\begin{enumerate}
  \item If $\X \in \mathbf{LAS}_{\underline{S}}$ has property $\mathbf{P}$, so is $\Phi_{log}\X$.
  \item If $\X\in\mathbf{FLAS}/{\underline{S}}$, then $\X$ has property $\mathbf{Q}$ if and only if $\Phi_{log}\X$ has proerty $\mathbf{Q}$. $\X$ is DM (quasi-compact, quasi-separated) if and only if $\Phi_{log}\X$ DM (quasi-compact, quasi-separated).
\end{enumerate}
\end{Proposition}
\begin{proof}
This result follows from the fact that if $U\rightarrow\X$ (where $U$ is log scheme) is a chart of $\X$, then $\Phi_{log} U\rightarrow\Phi_{log}\X$ is also a chart. If $\X\in\mathbf{FLAS}/\underline{S}$, then $U\rightarrow\X$ is a chart if and only if $\Phi_{log} U\rightarrow\Phi_{log}\X$ is a chart.
\end{proof}

In the end of this section we prove that an algebraic log stack always has enough compatible minimal objects.

\begin{Lemma}\label{factorbystrict_Stack}
$\X$ is a stack over $\mathbf{Flog}_S$, with an Asp-representable, strict, surjective, flat and locally of finite
  presentation morphism $U\rightarrow\X$ where $U$ is an algebraic log space. $f:T\rightarrow\X$ is a morphism from a fine log scheme $T$ to $\X$ over $\mathbf{Flog}_S$. Then $f$ factors through $\xymatrix{T\ar[r]^g & T_0 \ar[r]^h & \X}$ where $T_0$ is a fine log scheme, $\underline{g}=id$, $h$ is strict (i.e. $T_0\times_{\X}U\rightarrow U$ is strict). The factorization is unique in the following sense: if there is another factorization $\xymatrix{f':T\ar[r]^{g'} & T'_0 \ar[r]^{h'} & \X}$ with 2-isomorphism $\alpha:f\simeq f'$ where $\underline{g'}=id$, $h'$ is strict, then there is a unique pair $(u, \eta)$ where $u$ is a 1-automorphism $u:T_0\rightarrow T'_0$ with $g'=ug$, and $\beta$ is a 2-isomorphism $\beta:h\simeq h'u$ s.t. $g^\ast\beta=\alpha$. We call such factorization a strict factorization.
\end{Lemma}
\begin{proof}
First we describe some techniques for the proof:
\begin{enumerate}
  \item For a morphism $f:X\rightarrow Y$ between algebraic log spaces, we have a strict factorization as stated in \ref{factorbystrict_Aspace}. In fact, by theorem \ref{Phi_corres_Asp}, we will not distinguish fine log algebraic spaces and algebraic log spaces in the proof.
  \item Descent property of fine log structure in algebraic space (Remark below theorem \ref{Olsson_descent_fine_structure}).
  \item We will need descent techniques of morphisms from algebraic log spaces to stacks, as stated in Lemma \ref{descent_Asptostack}.
\end{enumerate}
\begin{description}
  \item[Existence] Consider the solid diagram
  $$\xymatrix{
  U_T\times_TU_T \ar[r]^u \ar@<1ex>[d] \ar@<-1ex>[d] & R \ar[r] \ar@<1ex>[d] \ar@<-1ex>[d] & U\times_\X U \ar@<1ex>[d] \ar@<-1ex>[d] \\
  U_T \ar[r]^v \ar[d] & V \ar[r] \ar@{-->}[d] & U \ar[d] \\
  T \ar@{-->}[r] \ar@/_/[rr] & T_0 \ar@{-->}[r]^h & \X
  }$$
  where the left vertical arrows come from the base change of the right vertical arrows. By Lemma \ref{Asprep}, $U\times_\X U$ is an algebraic log space. The first and second horizontal arrows are strict factorizations of algebraic log spaces. Since $\underline{u}=id$, $\underline{v}=id$, we have that $\xymatrix{\underline{R} \ar@<1ex>[r] \ar@<-1ex>[r] &\underline{V}}$ is effective with quotient $id:\underline{T}\rightarrow \underline{T_0}$. Moreover, since $\underline{V}\rightarrow \underline{T_0}$ is a flat, locally of finite presentation morphism and $\xymatrix{R \ar@<1ex>[r] \ar@<-1ex>[r] &V}$ are strict, we can descent the log structure on $V$ to $\underline{T_0}$ (\ref{Olsson_descent_fine_structure}). Denote this decent log scheme $T_0$, then $V\rightarrow T_0$ is strict, flat, locally of finite presentation. Since $\X$ is a stack, the descent data of morphisms $\xymatrix{R \ar@<1ex>[r] \ar@<-1ex>[r] &V \ar[r] & \X}$ gives $h:T_0\rightarrow \X$ fitting in the diagram.

  Now we prove $V\simeq T_0\times_{\X}U$ in the diagram. As a result, $h$ is strict. Consider the diagram:
  $$\xymatrix{
   & V \ar[d]^{i} \ar[rd] & \\
  U_T \ar[r]^{f} \ar[ru]^g & T_0\times_{\X}U \ar[r] & U
  }$$
  where $\underline{f}=\underline{g}=id$. Since $V\rightarrow U$ is strict, $i$ is an isomorphism, hence $V\simeq T_0\times_{\X}U$.

  This gives a strict factorization of $T\rightarrow\X$.

  \item[Uniqueness] Using the same diagram of another strict factorization:
  $$\xymatrix{
  U_T\times_TU_T \ar[r]^u \ar@<1ex>[d] \ar@<-1ex>[d] & R' \ar[r] \ar@<1ex>[d] \ar@<-1ex>[d] & U\times_\X U \ar@<1ex>[d] \ar@<-1ex>[d] \\
  U_T \ar[r]^v \ar[d] & V' \ar[r] \ar[d] & U \ar[d] \\
  T \ar[r]^{g'} \ar@/_/[rr]_{f'} & T'_0 \ar[r]^{h'} & \X
  }$$
  with a 2-isomorphism $\alpha:f\simeq f'$, $R'$ and $V'$ come from the pullbacks through $T'_0\rightarrow\X$. Then $R'$, $V'$ give the strict factorizations. By the uniqueness of the strict factorization of morphisms between algebraic log spaces, there are unique isomorphisms $r:R\simeq R'$ and $v:V\simeq V'$ compatible to the diagrams. We can descent them to an isomorphism $u:T_0\simeq T'_0$ compatible to the diagrams. Doing the same descent procedure, we get a 2-isomorphism $\beta:h\simeq h'u$, s.t. $g'^\ast\beta=\alpha$. The uniqueness of $u$ and $\beta$ comes from chasing the diagram.
\end{description}
\end{proof}

\begin{Theorem}\label{alglogstackminimal}
Suppose that $\X$ is a stack over $\mathbf{Flog}_S$, with a Asp-representable, strict, surjective, flat and locally of finite
  presentation morphism $U\rightarrow\X$ where $U$ is an algebraic log space. Then $\X$ has enough compatible minimal objects.
\end{Theorem}
\begin{proof}
Let the subcategory $\X_m$ of $\X$ consist of objects corresponding to a strict morphisms $T\rightarrow X$ (i.e. $T\times_{\X}U\rightarrow U$ is strict) where $T$ is a fine log scheme. Then $\X_m$ form a compatible system of minimal objects due to Lemma \ref{factorbystrict_Stack}, by the same argument as in theorem \ref{alglogstackminimalweak}.
\end{proof}

As a corollary, all algebraic log stacks have enough compatible minimal objects. Hence we get the representation theorem of algebraic log stacks:
\begin{Theorem}\label{Phi_corres_Astack}
$\Phi_{log}$ sends $\mathbf{LAS}_{\underline{S}}$ to $\mathbf{ALS}_S$, and restricts on $\mathbf{FLAS}_{\underline{S}}$ to be a strict 2-equivalence:
$$\mathbf{FLAS}_{\underline{S}}\rightarrow\mathbf{ALS}_S$$
$\Phi_{log} \X\simeq\Phi_{log} \Y$ if and only if $\X^{int}\simeq \Y^{int}$.
\end{Theorem}
\begin{proof}
By proposition \ref{Phialgstack}, $\Phi_{log}$ sends $\mathbf{LAS}_{\underline{S}}$ to $\mathbf{ALS}_S$, and restricts to strict fully faithful functor $\mathbf{FLAS}_{\underline{S}}\rightarrow\mathbf{ALS}_S$. By Theorem \ref{alglogstackminimal} and Lemma \ref{GC}, the essential images of $\Phi_{log}$ are $\mathbf{ALS}_S$. Hence $$\mathbf{FLAS}_{\underline{S}}\rightarrow\mathbf{ALS}_S$$ is a strict 2-equivalence.
The last part follows directly from Lemma \ref{integralpart_stack}.
\end{proof}

\section{Applications}
Due to the correspondence $\ref{Phi_corres_Astack}$, we get the results of algebraic log stacks from the known results in algebraic stacks. In this section we list some of the fundamental ones:
\subsection{Bootstrapping algebraic log stacks}
\begin{Theorem}\label{fppfrepofstack}
Let S be a fine log scheme. Let $F:\X\rightarrow \Y$ be a 1-morphism of stacks over $\mathbf{Flog}_{S,fppf}$ . If
\begin{enumerate}
  \item $\X$ is an algebraic log space,
  \item $F$ is Asp-representable, strict, surjective, flat and locally of finite
  presentation,
\end{enumerate}
then $\Y$ is an algebraic log stack.
\end{Theorem}
\begin{proof}
By theorem \ref{alglogstackminimal}, $\Y$ has enough minimal objects. So there is $\Y'\in \mathbf{FLogCFG}_{\underline{S}}$ satisfying that $\Phi_{log}\Y'=\Y$. By theorem \ref{Phi_corres_Asp} and Lemma \ref{AB_Phi}, the morphism $f$ descents to a $f':\X'\rightarrow\Y'$ in $\mathbf{FLogCFG}_{\underline{S}}$, $\X=\Phi_{log}\X'$, $f=\Phi_{log} f'$. By Proposition \ref{Phi_premorph_Astack}, $f'$ is Asp-representable, strict, surjective, flat and locally of finite
  presentation, which implies that $\underline{\Y'}$ is an algebraic stack (\cite{dejong} Theorem 70.16.1). Hence $\Y'\in \mathbf{FLAS}_S$, and $\Y$ is an algebraic log stack.
\end{proof}

\subsection{Groupoid presentation of algebraic log stack}\

For algebraic log stacks, there is also the notion of presentation by groupoid, a groupoid in algebraic log spaces is $(U; R; s; t; c)$ where $U$ and $R$ are algebraic log spaces, $s$, $t$, $c$ are strict, and $(\underline{U}; \underline{R}; \underline{s}; \underline{t}; \underline{c})$ (see Definition \ref{DEF_underlinespace-AlgL}) is groupoid in algebraic space.

Let $(U; R; s; t; c)$ be a groupoid in log algebraic spaces, $\pi:[\underline{U}/ \underline{R}]\rightarrow\mathbf{Sch}_{\underline{S}}$ is the associate stack. Since $\mathbf{Flog}_S$ is stack over $\mathbf{Sch}_{\underline{S}}$, $\pi$ factor through $\mathbf{Flog}_S$. This gives $[\underline{U}/ \underline{R}]$ a fine log structure. We denote this stack with fine log structure $[U/R]$. By abstract nonsense we have $\Phi_{log}[U/R]=\Phi_{log} U/\Phi_{log} R$.

\begin{Theorem}\label{stackgroupiod}
Let $S$ be a fine log scheme and $X$ be an algebraic log stack over $S$. $f : U\rightarrow X$ is a surjective strict log smooth morphism where $U$ is an algebraic log space over $S$. Let $(U; R; s; t; c)$ be the associated groupoid in log algebraic spaces and $f_{can}:[U/R]\rightarrow X$ be the associated map. Then
\begin{enumerate}
  \item the morphisms $s$, $t$ are strict log smooth, and
  \item the 1-morphism $f_{can}:[U/R]\rightarrow X$ is an equivalence.
\end{enumerate}
\end{Theorem}\noindent
\textbf{Remark:} If the morphism $f:U\rightarrow X$ is only assumed surjective, strict, flat and locally of finite presentation, then it is still the case that $f_{can}:[U/R]\rightarrow X$ is an equivalence. In this case the morphisms $s$, $t$ are strict, flat and locally of finite presentation, but not smooth in general.
\begin{proof}
By theorem \ref{Phi_corres_Astack}, the result is direct from (\cite{dejong}, Lemma 67.16.2) and the descent result on log structures (Remark below theorem \ref{Olsson_descent_fine_structure}).
\end{proof}

\begin{Theorem}\label{groupoidstack}
Let $S$ be a fine log scheme and $(U; R; s; t; c)$ be a strict log smooth groupoid of algebraic log spaces over $S$. Then the quotient stack $[U/R]$ is an algebraic log stack over $S$.
\end{Theorem}
\begin{proof}
By theorem \ref{Phi_corres_Astack}, we may assume that $(U; R; s; t; c)=\Phi_{log}(U_0; R_0; s_0; t_0; c_0)$ comes from the groupoid of fine log algebraic spaces. It is directly from (\cite{dejong}, Theorem 67.17.3) and the descent result on fine log structures (Remark below theorem \ref{Olsson_descent_fine_structure}) that $[U_0/R_0]$ is a log algebraic stack. Hence $[U/R]$ is an algebraic log stack.
\end{proof}

\subsection{DM = unramified diagonal}
\begin{Theorem}\label{DM=unramify_diagonal}
An algebraic log stack $\X$ is DM if and only if the diagonal $\Delta_{\X}$ is unramified.
\end{Theorem}
\begin{proof}
By theorem \ref{Phi_corres_Astack} we can assume that $\X=\Phi_{log}\X'$ for a fine log algebraic stack $\X'$. Since $\X$ is DM if and only if $\X'$ is DM (theorem \ref{Phi_preproperty_Astack}) and the latter is equivalent to that $\Delta_{\X'}$ is unramified, it's enough to show that $\Delta_{\X'}$ is unramified if and only if $\Delta_{\X}$ is unramified.

Consider the decomposition $$\xymatrix{\Delta_{\X'}: \X'\ar[r]^{\delta} & (\X'\times\X')^{int} \ar[r]^\tau & \X'\times\X'}$$. We know from theorem \ref{Phi_premorph_Astack} that $\delta$ is unramified if and only if $\Delta_{\X}$ is unramified. Since  $\tau$ is a closed immersion we have that $\delta$ is unramified if and only if $\Delta_{\X'}$ is unramified.
\end{proof}

\begin{Corollary}\label{Algspa=mono_diagonal}
For an algebraic log stack $\X$, the following are equivalent:
 \begin{enumerate}
   \item $\X$ is an algebraic log space;
   \item for every $x\in\X_U$, where $U\in\mathbf{Flog}_S$, $Aut_{\X_U}(x)=\{id_x\}$;
   \item the diagonal $\Delta_{\X}:\X\rightarrow\X\times_{\X}\X$ is fully faithful.
 \end{enumerate}
\end{Corollary}
\begin{proof}
($1\Rightarrow 2$) is clear. By abstract nonsense we know that $(2\Leftrightarrow 3)$.

($3\Rightarrow 1$): Since the diagonal $\Delta_{\X}:\X\rightarrow\X\times_{\X}\X$ is fully faithful, it is unramified. Hence $\X$ has an \'etale cover by fine log scheme. It remains to show that $\Delta_{\X}:\X\rightarrow\X\times_{\X}\X$ is representable. Given any $U\rightarrow \X\times_{\X}\X$, where $U$ is a fine log scheme. $\X\times_{\X\times_{\X}\X}U\rightarrow U$ is separate, quasi-finite, locally of finite presentation since it is fully faithful. So $\X\times_{\X\times_{\X}\X}U\rightarrow U$ is quasi-affine and $\X\times_{\X\times_{\X}\X}U$ is a fine log scheme.
\end{proof}

\begin{Corollary}\label{criterion_Asprep}
A morphism $f:\X\rightarrow\Y$ in $\mathbf{ALS}_S$ is Asp-representable if and only if $\Delta_f:\X\rightarrow\X\times_{\Y}\X$ is a monomorphism.
\end{Corollary}

\end{document}